\documentclass[12pt]{article}
\usepackage{amsmath}
\usepackage{amssymb}

\usepackage{mathtools}
\usepackage{amsthm}
\usepackage[usenames]{color}

\usepackage{graphicx}

\newcommand{\1}[1]{{\mathbf 1}{\{#1\}}}

\newcommand{\vr}{\varrho}

\newcommand{\eps}{\varepsilon}
\newcommand{\Z}{{\mathbb Z}}

\newcommand{\V}{{\mathcal V}}

\newcommand{\M}{{\mathcal M}}

\newcommand{\G}{{\mathcal G}}

\newcommand{\R}{{\mathbb R}}

\newcommand{\T}{{\mathcal T}}

\newcommand{\s}{{\widehat S}}

\let\phi=\varphi

\newcommand{\E}{{\mathbb E}}

\newcommand{\tX}{{\widetilde X}}

\newcommand{\diam}{{\mathop{\mathrm{diam}}}}
\newcommand{\8}{{\infty}}

\newcommand{\IP}{{\mathbb P}}
\newcommand{\IE}{{\mathbb E}}

\newcommand{\hP}{\widehat{P}}

\newcommand{\hE}{\widehat{E}}

\newcommand{\capa}{\mathop{\mathrm{cap}}}

\newcommand{\Cor}{\mathop{\mathrm{Cor}}}
\newcommand{\Var}{\mathop{\mathrm{Var}}}
\newcommand{\hcapa}{\mathop{\widehat{\mathrm{cap}}}}
\newcommand{\hm}{\mathop{\mathrm{hm}}\nolimits}

\newcommand{\dist}{\mathop{\mathrm{dist}}}

\newcommand{\htau}{\widehat{\tau}}

\allowdisplaybreaks

\newtheorem{theo}{Theorem}[section]
\newtheorem{lem}[theo]{Lemma}
\newtheorem{df}[theo]{Definition}
\newtheorem{prop}[theo]{Proposition}

\newtheorem{rem}[theo]{Remark}

\title{Two-dimensional random interlacements and late points 
for random walks}
\author{Francis Comets$^{1}$ \and
 Serguei~Popov$^{2}$ \and Marina Vachkovskaia$^{2}$}

\begin{document}

\maketitle

{\footnotesize 
\noindent $^{~1}$Universit\'e Paris Diderot -- Paris 7, 
Math\'ematiques, 
 case 7012, F--75205 Paris
Cedex 13, France
\\
\noindent e-mail:
\texttt{comets@math.univ-paris-diderot.fr}

\noindent $^{~2}$Department of Statistics, Institute of Mathematics,
 Statistics and Scientific Computation, University of Campinas --
UNICAMP, rua S\'ergio Buarque de Holanda 651,
13083--859, Campinas SP, Brazil\\
\noindent e-mails: \texttt{\{popov,marinav\}@ime.unicamp.br}

}

\begin{abstract}
We define the model of two-dimensional random interlacements
using simple random walk trajectories conditioned on never hitting
the origin, and then obtain some properties of this model. 
Also, for random walk on a large torus conditioned on not 
hitting the origin up to some time proportional 
to the mean cover time, we show that 
the law of the vacant set around the origin is close
to that of random interlacements at the corresponding level. 
Thus, this new model provides a way to understand the structure 
of the set of late points of the covering process from a
microscopic point of view.
 \\[.3cm]\textbf{Keywords:} random interlacements, hitting time, 
 simple random walk, Doob's $h$-transform
\\[.3cm]\textbf{AMS 2010 subject classifications:}
Primary 60K35. Secondary 60G50, 82C41.
 
\end{abstract}

\section{Introduction}
\label{s_intro}

We start by an informal description of our purpose.

%

\subsection{Random interlacements in two dimensions}
\label{s_RI}
Random interlacements were introduced by Sznitman in~\cite{Szn10},
motivated by the problem of disconnection of the discrete
torus $\Z_n^d:=\Z^d/n\Z^d$ by the trace of simple random walk,
in dimension~$3$ or higher. Detailed accounts can be found in the survey~\cite{CT12} and
the recent book~\cite{DRS14}.
Loosely speaking, the model of random interlacements
in~$\Z^d$, $d\geq 3$, is a stationary 
Poissonian soup of (transient) doubly infinite simple random
walk trajectories on the integer lattice. There is an additional 
parameter~$u>0$ entering the intensity measure of the Poisson 
process, the larger $u$~is the more trajectories are thrown in.
The sites of~$\Z^d$ that are not touched by the trajectories constitute
the \emph{vacant set}~$\V^u$.
The random interlacements are constructed simultaneously for all $u>0$
in such a way that $\V^{u_1}\subset \V^{u_2}$ if $u_1>u_2$.
In fact, the law of the vacant set at level~$u$ can be
uniquely characterized by the following identity:
\begin{equation}
\label{eq_vacant>3}
 \IP[A\subset \V^u] = \exp\big(-u \capa(A)\big)
\quad \text{for all finite $A\subset\Z^d$},
\end{equation}
where $\capa(A)$ is the \emph{capacity} of~$A$.
Informally, the capacity measures how ``big'' the set
 is from the point of view of the walk, see Section~6.5
of~\cite{LL10} for formal definitions, 
as well as~\eqref{df_eq_measure}--\eqref{df_cap_trans} below.

At first glance, the title of this section seems to be meaningless,
just because even a single trajectory of two-dimensional
simple random walk a.s.\ visits all sites of~$\Z^2$, so the vacant
set would be always empty. Nevertheless, there is also a natural notion of capacity in two dimensions (cf.\ Section~6.6
of~\cite{LL10}), 
 so one may wonder if there is a way to construct
 a decreasing family $(\V^\alpha, \alpha>0)$ of random subsets
 of~$\Z^2$ in such a way that
a formula analogous to~\eqref{eq_vacant>3} 
holds for every finite~$A$. This is, however, clearly not possible
since the two-di\-men\-sional capacity of one-point sets equals~$0$.
On the other hand, it turns out to be possible to construct 
such a family so that 
\begin{equation}
\label{eq_vacant2}
 \IP[A\subset \V^\alpha] = \exp\big(-\pi\alpha \capa(A)\big)
\end{equation}
 holds \emph{for all sets containing the origin}
(the factor~$\pi$ in the exponent is just for 
convenience, as explained below).
We present this construction in Section~\ref{s_results}.
To build the interlacements, 
we use trajectories of simple random walks 
conditioned on never hitting the origin. 
Of course, the law of the vacant set is no longer translationally
invariant, but we show that it has the property of
\emph{conditional} translation invariance, 
cf.\ Theorem~\ref{t_properties_RI} below.
In addition, we will see that (similarly to the $d\geq 3$
case) the random object we construct has strong connections
to random walks on two-dimensional torus. All this makes 
us believe that ``two-dimensional random interlacements'' is 
the right term for the object we introduce in this paper.


\subsection{Cover time and late points of simple random walk on 
a discrete torus}
\label{s_late}
Consider the simple random walk on the two-dimensional 
discrete torus $\Z^2_n$ with the starting point chosen 
uniformly at random. Let~$\T_n$ be the first moment
when this random walk visits all sites of~$\Z^2_n$;
we refer to~$\T_n$ as the \emph{cover time}
of the torus. It was shown in~\cite{DPRZ04}
that $\frac{\T_n}{n^2\ln^2 n}\to \frac{4}{\pi}$
in probability; later, this result was refined
in~\cite{D12}, and then even finer results on the first correction to this limit were
obtained in~\cite{BK14} for the similar problem
of covering the (continuous) torus with a Brownian sausage.

The structure of the set of \emph{late points} 
(i.e., the set of points that are still unvisited up to
a given time) of the random walk 
on the torus is rather well understood in dimensions $d\geq 3$,
see~\cite{B13,MS13}, 
and also~ \cite{GdH14} for the continuous case.
 On the other hand, much remains to be
discovered in two dimensions. 
After the heuristic arguments of~\cite{BH91} revealing an intriguing random set,
 it was shown
in~\cite{DPRZ06} that this set has interesting fractal-like properties when the 
elapsed time is a fraction of the expected cover time. This particular behaviour is induced by 
long distance correlations between hitting times due to recurrence.
In this paper, we prove that the law of the uncovered set 
around the origin at time $\frac{4\alpha}{\pi}n^2\ln^2 n$
\emph{conditioned} on the event that the origin is uncovered,
is close to the law of two-dimensional random
interlacements at level~$\alpha$ (Theorem~\ref{t_conditional}). 
We hope that this result
will lead to other advances in understanding the structure
of the uncovered set. 

We now explain why conditioning is necessary to observe a 
meaningful point process.
In two dimensions, if we know that simple random walk has 
visited a given site by a large time, then it is likely, 
by recurrence, that it has  has 
visited all the nearby sites as well. 
This means that a fixed-size window around the origin
will be typically either full (on the event that the origin
was visited) 
or empty (if the origin was not yet visited). 
Therefore, we need to
condition on a rare event to obtain a nontrivial limit.


As a side note, observe that the two-dimensional random interlacements 
relate to the simple random walk on the torus at a time 
proportional to the cover time. In higher dimensions, one starts to observe the 
``interlacement regime'' already at times below the cover time
by a factor of~$\ln n$.
\medskip

\textbf{Organisation of the paper}: In Section~\ref{s_results}
we construct the model of random interlacements and present some of its properties. 
In Section~\ref{s_rw_interl} we formulate a result 
relating this model and the vacant 
set of the simple random walk on the discrete torus.
We prove some results on the spot
-- when short arguments are available -- postponing the proof 
of the other ones to Section~\ref{s_proofs}. 
 Section~\ref{s_aux} contains 
a number of auxiliary facts needed for the proof of the main results.

\section{Definitions and results}
\label{sec:def-res}

We start by defining the two-dimensional 
 random interlacement process, which involves
some potential-theoretic considerations. 

\subsection{Random interlacements: definitions, properties}
\label{s_results}
Let~$\|\cdot\|$ be the Euclidean norm. Define the (discrete)
ball
\[
 B(x,r) = \{y\in \Z^2: \|y-x\|\leq r\}
\]
(note that $x$ and $r$ need not be integer), and
abbreviate $B(r):=B(0,r)$.
We write~$x\sim y$ if $x$ and $y$ are neighbours on~$\Z^2$.
The (internal) boundary of $A\subset\Z^2$ is defined by
\[
 \partial A = \{x\in A: \text{there exists }y\in \Z^2\setminus A
 \text{ such that }x\sim y\}.
\]
Let~$(S_n, n\geq 0)$ be two-dimensional simple
random walk. Write~$P_x$ for the law of the walk started from~$x$
and~$E_x$ for the corresponding expectation.
Let
\begin{align}
\tau_0(A) &= \inf\{k\geq 0: S_k\in A\} \label{entrance_t},\\
\tau_1(A) &= \inf\{k\geq 1: S_k\in A\} \label{hitting_t}
\end{align}
be the entrance and the hitting time of the set~$A$ by 
simple random walk~$S$ (we use the convention $\inf \emptyset = +\8$). 
For a singleton $A=\{x\}$, we will write $\tau_i(A)=\tau_i(x)$, $i=0,1$,
for short. 
Define the potential kernel~$a$ by
\begin{equation}
\label{def_a(x)}
a(x) = \sum_{k=0}^\infty\big(P_0[S_k\!=\!0]-P_x[S_k\!=\!0]\big).
\end{equation}
It can be shown that the above series indeed converges 
and we have~$a(0)=0$, $a(x)>0$ for $x\neq 0$, and
\begin{equation}
\label{formula_for_a}
 a(x) = \frac{2}{\pi}\ln \|x\| + \gamma' + O(\|x\|^{-2}) 
\end{equation}
as $x\to\infty$, cf.\ Theorem~4.4.4 of~\cite{LL10}
(the value of $\gamma'$ is 
known\footnote{$\gamma'=\pi^{-1}(2\gamma+\ln 8)$, 
where $\gamma=0.5772156\dots$ is the Euler-Mascheroni constant}, 
but we will not need it in
this paper).
Also, the function~$a$ is harmonic outside the origin, i.e.,
\begin{equation}
\label{a_harm}
 \frac{1}{4}\sum_{y: y\sim x}a(y) = a(x) \quad \text{ for all }
 x\neq 0.
\end{equation}
Observe that~\eqref{a_harm} immediately implies 
that $a(S_{k\wedge \tau_0(0)})$ is a martingale, we will
repeatedly use this fact in the sequel.
With some abuse of notation, we also consider
the function 
\[
a(r)=\frac{2}{\pi}\ln r + \gamma' 
\]
of a \emph{real}
argument~$r\geq 1$. The advantage of using this notation is e.g.\ that, 
due to~\eqref{formula_for_a}, we may write, as $r\to\infty$,
\begin{equation}
\label{real_a}
 \sum_{y\in\partial B(x,r)} \nu(y)a(y) = a(r) 
+ O\Big(\frac{\|x\|\vee 1}{r}\Big)
\end{equation}
for \emph{any} probability measure~$\nu$ on $\partial B(x,r)$.
The \emph{harmonic measure} of a finite $A\subset\Z^2$
is 
the entrance law ``starting at infinity'',
\begin{equation}
\label{def_hm}
 \hm_A(x) = \lim_{\|y\|\to\infty}P_y[S_{\tau_1(A)}=x].
\end{equation}
The existence of the above limit
follows from Proposition~6.6.1 of~\cite{LL10}; also, this 
proposition together with~(6.44) implies that
\begin{equation}
\label{hm_escape}
\hm_A(x) = \frac{2}{\pi}\lim_{R\to \infty} 
P_x\big[\tau_1(A)>\tau_1\big(\partial B(R)\big)\big]\ln R .
\end{equation}
Intuitively, \eqref{hm_escape} means that
the harmonic measure at $x\in\partial A$
is proportional to the probability of escaping from~$x$
to a large sphere. Observe also that,
by recurrence of the walk, $\hm_A$ is a probability measure
on~$\partial A$. 
Now, for a finite set~$A$ containing the origin,
we define its capacity by
\begin{equation}
\label{df_cap2}
\capa(A) = \sum_{x\in A}a(x)\hm_A(x);
\end{equation}
in particular, $\capa\big(\{0\}\big)=0$ since $a(0)=0$.
For a set not containing the origin, its capacity is defined
as the capacity of a translate of this set that does contain
the origin. Indeed, it can be shown that the capacity does not depend
on the choice of the translation.
A number of alternative definitions are available,
cf.\ Section~6.6 of~\cite{LL10}. 
Intuitively, the capacity of $A$ represents the difference in size between the set $A$ and a single point, as seen from infinity:
From (6.40) and the formula above Proposition 6.6.2 in~\cite{LL10}, we can write
\[
P_x\big[\tau_1\big(\partial B(r)\big) <\tau_1(A)\big] 
= \frac{ \ln \|x\| + \frac{\pi \gamma'}{2} - \frac{\pi}{2} \capa(A) 
 + \eps_A(x) + \eps_{A,x}'(r)}{\ln r},
\]
where $\eps_A(x)$ vanishes  as $\|x\| \to \8$ and $\eps_{A,x}'(r)$ vanishes  as $r \to \8$ keeping fixed the other variables.
Observe that, by symmetry,
the harmonic measure of any two-point set is uniform,
so $\capa\big(\{x,y\}\big)=\frac{1}{2}a(y-x)$ for any $x,y\in\Z^2$.
Also, \eqref{real_a} implies that
\begin{equation}
\label{capa_ball}
 \capa\big(B(r)\big) = a(r) + O(r^{-1}).
\end{equation}

Let us define another random walk $(\s_n, n\geq 0)$
on~$\Z^2$ (in fact, on~$\Z^2\setminus \{0\}$) in the following way:
the transition probability from~$x$ to~$y$ 
equals $\frac{a(y)}{4a(x)}$ for all $x\sim y$
(this definition does not make sense for $x=0$, but this is 
not a problem since the walk~$\s$ can never enter the origin anyway). The 
walk~$\s$ can be thought of as the Doob $h$-transform
of the simple random walk, under condition of not hitting the origin
(see Lemma~\ref{l_relation_S_hatS} and its proof).
Note that~\eqref{a_harm} implies that the random walk~$\s$
is indeed well defined, and, clearly, it is 
an irreducible Markov chain on~$\Z^2\setminus \{0\}$. We 
denote by $\hP_x, \hE_x$ the probability and expectation
for the random walk~$\s$ started from~$x \neq 0$.
Let $\htau_0, \htau_1$ be defined as in 
\eqref{entrance_t}--\eqref{hitting_t}, but with~$\s$ in the place of~$S$.
Then, it is straightforward to observe that
\begin{itemize}
 \item the walk~$\s$ is reversible, with the reversible
 measure~$\mu_x:=a^2(x)$;
 \item in fact, it can be represented as a random walk
 on the two-dimensional lattice with conductances (or weights)
 $\big(a(x)a(y), x,y\in \Z^2, x\sim y\big)$;
 \item $\big(a(x), x\in \Z^2\setminus \{0\}\big)$ is an
 excessive measure for~$\s$  (i.e., for all $y \neq 0$, $\sum_x a(x) \hP_x(\s_1=y) \leq a(y)$), with equality failing at the 
 four neighbours
 of the origin. Therefore, by e.g.\
 Theorem~1.9 of Chapter~3 of~\cite{R84}, the random walk~$\s$
 is transient;
 \item an alternative argument for proving transience is
the following: let~$\mathcal{N}$ be the set of the four
neighbours of the origin. Then, a direct calculation 
shows that $1/a(\s_{k\wedge \htau_0(\mathcal{N})})$ is a 
martingale. The transience then follows from Theorem~2.2.2 of~\cite{FMM}.
\end{itemize}

Our next definitions are appropriate for the transient case.
For a finite~$A\subset \Z^2$,
we define the \emph{equilibrium measure} 
\begin{equation}
\label{df_eq_measure}
 \widehat e_A(x) = \1{x\in A} \hP_x[\htau_1(A)=\infty]\mu_x,
\end{equation}
and the capacity (with respect to~$\s$) 
\begin{equation}
\label{df_cap_trans}
 \hcapa(A) = \sum_{x\in A}\widehat e_A(x).
\end{equation}
Observe that, since $\mu_0=0$, it holds that
$\hcapa(A)=\hcapa(A\cup\{0\})$ for any set~$A\subset \Z^2$.

Now, we use the general construction of random interlacements
on a transient weighted graph introduced in~\cite{T09}.
In the following few lines we briefly summarize this
construction.
Let~$W$ be the space of all doubly infinite nearest-neighbour 
transient trajectories in~$\Z^2$,
\begin{align*}
 W =& \big\{\vr=(\vr_k)_{ k\in \Z}: 
\vr_k\sim \vr_{k+1} \text{ for all }k;\\
&~~~~~~~~~~\text{ the set }
 \{m: \vr_m=y\} \text{ is finite for all }y\in\Z^2 \big\}.
\end{align*}
We say that~$\vr$ and~$\vr'$ are equivalent if they 
coincide after a time shift, i.e., $\vr\sim\vr'$
when there exists~$k$ such that $\vr_{m+k}=\vr_m$ for all~$m$.
Then, let $W^*=W/\sim$ be the space of trajectories
modulo time shift, and define~$\chi^*$ to be the canonical
projection from~$W$ to~$W^*$. For a finite $A\subset \Z^2$, 
let~$W_A$ be the set of trajectories in~$W$ that intersect~$A$,
and we write~$W^*_A$ for the image of~$W_A$ under~$\chi^*$.
One then constructs the random interlacements as Poisson
point process on $W^*\times \R^+$ with the intensity measure
$\nu\otimes du$, where~$\nu$ is described in the following
way. It is the unique sigma-finite measure on  the cylindrical sigma-field  of~$W^*$
such that for every finite~$A$
\[
 \mathbf{1}_{W^*_A} \cdot \nu = \chi^* \circ Q_A,
\]
where the finite measure~$Q_A$ on~$W_A$ is determined by the
following equality:
\[
Q_A\big[(\vr_k)_{k\geq 1}\!\in \!F, \vr_0\!=\!x, (\vr_{-k})_{k\geq 1}\!\in \!G\big]
=  \widehat e_A(x) \cdot \hP_x[F] \cdot \hP_x[G\mid \htau_1(A)\!=\!\infty].
\]
The existence and uniqueness of~$\nu$ was shown in 
Theorem~2.1 of~\cite{T09}.

\begin{df} \label{def:ri}
For a configuration $\sum_{\lambda}\delta_{(w^*_\lambda,u_\lambda)}$
of the above Poisson process, the {process of random interlacements
at level~$\alpha$} (which will be referred to as RI($\alpha$))
is defined as the set of trajectories with label less than or equal 
to~$\pi\alpha$, i.e.,
\[
 \sum_{\lambda: u_\lambda\leq \pi\alpha}  
 \delta_{w^*_\lambda} \;.
\]
\end{df}

Observe that this definition is somewhat unconventional (we used~$\pi\alpha$
instead of just~$\alpha$, as one would normally do), but we will
see below that it is quite reasonable in two dimensions,
since the formulas become generally cleaner.

It is important to have in mind the following ``constructive''
description of random interlacements at level~$\alpha$ ``observed''
on a finite set $A\subset \Z^2$. Namely,
\begin{itemize}
 \item take a Poisson($\pi\alpha\hcapa(A)$) number of particles;
 \item place these particles on the boundary of~$A$
 independently, with distribution
 $\overline{e}_A = \big((\hcapa A)^{-1}\widehat e_A(x), x\in A\big)$;
 \item let the particles perform independent $\s$-random walks
 (since~$\s$ is transient, each walk only leaves a finite trace
 on~$A$).
\end{itemize}

It is also worth mentioning that the FKG inequality holds
for random interlacements, cf.\ Theorem~3.1 of~\cite{T09}.

The \emph{vacant set} at level $\alpha$,
\[
 \V^\alpha = \Z^2 \setminus \bigcup_{\lambda: u_\lambda \leq \pi\alpha} \omega^*_\lambda
 (\Z),
\]
is the set of lattice points not covered by the random interlacement. It contains the origin by definition. In Figure~\ref{f_simulation} we present
a simulation of the vacant set for different values of the 
parameter.
\begin{figure}
\begin{center}
\includegraphics[width=0.7\textwidth]{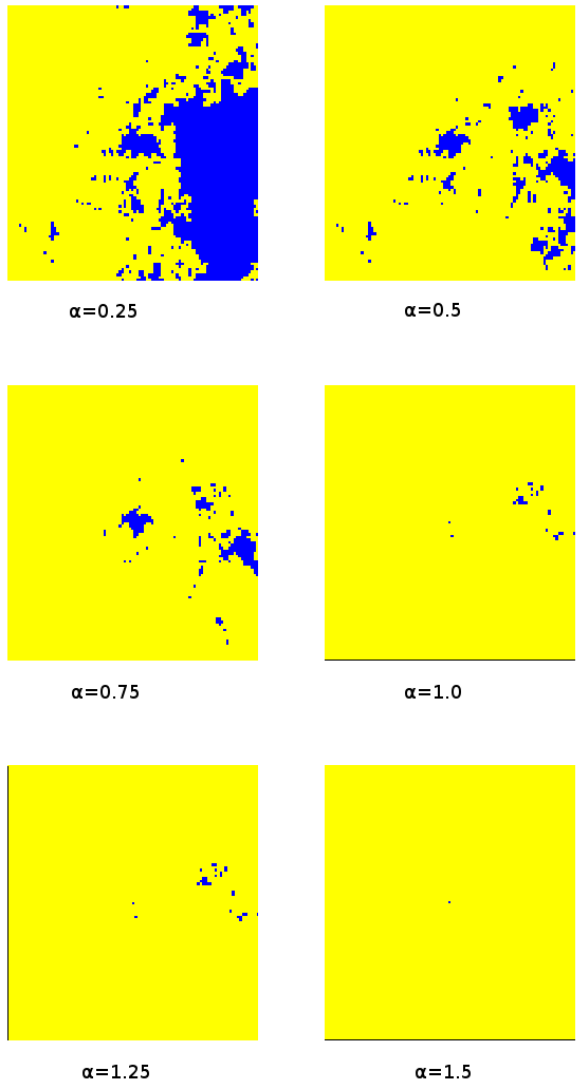}
\caption{A realization of the vacant set (dark blue) of RI($\alpha$) 
for different values of~$\alpha$.
For $\alpha=1.5$ the only vacant site is the origin.
Also, note that we see the same neighbourhoods 
of the origin for $\alpha=1$ and~$\alpha=1.25$;
this is not surprising since just a few new walks enter the picture when 
increasing the rate by a small amount.}
\label{f_simulation}
\end{center}
\end{figure}

As a last step, we need to show that we have indeed 
constructed the object for which~\eqref{eq_vacant2} 
is verified. For this, we need to prove the following fact:
\begin{prop}
\label{p_equalcapa}
 For any finite set $A\subset \Z^2$ such that $0\in A$ it holds that
$\capa(A)=\hcapa(A)$.
\end{prop}

\begin{proof}
 Indeed, consider an arbitrary~$x\in\partial A$, $x\neq 0$,
and (large)~$r$ such that $A\subset B(r-2)$. 
Write using~\eqref{formula_for_a}
\begin{align*}
 \hP_x\big[\htau_1(A)>\htau_1\big(\partial B(r)\big)\big] & =
   \sum_\vr \frac{a(\vr_{\text{end}})}{a(x)}
   \Big(\frac{1}{4}\Big)^{|\vr|}\\
   &= \big(1+o(1)\big)\frac{\frac{2}{\pi}\ln r}{a(x)}
   \sum_\vr \Big(\frac{1}{4}\Big)^{|\vr|}\\
   &= \big(1+o(1)\big)\frac{\frac{2}{\pi}\ln r}{a(x)}\,
    P_x\big[\tau_1(A)>\tau_1\big(\partial B(r)\big)\big],
\end{align*}
where the sums are taken over all trajectories~$\vr$ 
that start at~$x$, end at~$\partial B(r)$, and avoid
$A\cup \partial B(r)$ in between;
$\vr_{\text{end}}\in \partial B(r)$ stands for the 
ending point of the trajectory, and~$|\vr|$ is the 
trajectory's length.
Now, we send~$r$ to infinity and use~\eqref{hm_escape}
to obtain that, if $0\in A$, 
\begin{equation}
\label{escape_identity}
a(x) \hP_x[\htau_1(A)=\infty] = \hm_A(x).
\end{equation}
Multiplying by~$a(x)$ and summing over
 $x\in A$ (recall that $\mu_x=a^2(x)$) we 
obtain the expressions in~\eqref{df_cap2} and~\eqref{df_cap_trans}
and thus conclude the proof.
\end{proof}

Together with formula~(1.1) of~\cite{T09}, Proposition~\ref{p_equalcapa} shows the fundamental relation~(\ref{eq_vacant2})
announced in introduction:
for all finite subsets~$A$ of $\Z^2$ containing the origin,
\[
\IP[ A \subset \V^\alpha] = \exp\big(-\pi\alpha \capa(A)\big).
\] 
As mentioned before, the law of two-dimensional random interlacements
is not translationally invariant, 
although it is of course invariant with respect to reflections/rotations
of~$\Z^2$ that preserve the origin.
Let us describe some other basic properties of two-dimensional
random interlacements:
\begin{theo}
\label{t_properties_RI}
\begin{itemize}
 \item[(i)] For any $\alpha>0$, $x\in\Z^2$,
 $A\subset \Z^2$, it holds that 
\begin{equation}
\label{properties_RI_i}
 \IP[A\subset\V^\alpha \mid x\in \V^\alpha]=
  \IP[-A+x\subset\V^\alpha \mid x\in \V^\alpha].
\end{equation}
More generally, for all  $\alpha>0$, $x\in\Z^2 \setminus \{0\}$,
 $A\subset \Z^2$, and any
 lattice isometry~$M$ exchanging $0$ and $x$, we have 
\begin{equation}
\label{properties_RI_i'}
 \IP[A\subset\V^\alpha \mid x\in \V^\alpha]=
  \IP[MA \subset\V^\alpha \mid x\in \V^\alpha].
\end{equation}
 \item[(ii)] With $\gamma'$ from (\ref{formula_for_a}) we have
\begin{equation}
\label{properties_RI_ii}
\IP[x\in \V^\alpha]=\exp\Big(-\pi\alpha \frac{a(x)}{2}
\Big)
=e^{-\gamma'\pi\alpha/2}\|x\|^{-\alpha}\big(1+O(\|x\|^{-2})\big).
\end{equation}
 \item[(iii)] For~$A$ such that $0\in A\subset B(r)$
and $x\in\Z^2$ such that $\|x\|\geq 2r$ we have
\begin{equation}
\label{properties_RI_iii}
 \IP[A\subset\V^\alpha \mid x\in \V^\alpha]= 
   \exp\Bigg(-\frac{\pi\alpha}{4}\capa(A)
\frac{1+O\big(\frac{r\ln r \ln\|x\|}{\|x\|}\big)}
{1-\frac{\capa(A)}{2a(x)}
+O\big(\frac{r\ln r}{\|x\|}\big)}
\Bigg).
\end{equation}
 \item[(iv)] For $x,y\neq 0$, $x\neq y$, we have 
$\IP\big[\{x,y\}\subset \V^\alpha\big] 
= \exp\big(-\pi \alpha\Psi\big)$,
where 
\[
  \Psi = \frac{a(x)a(y)a(x-y)}
 {a(x)a(y)+a(x)a(x-y)+a(y)a(x-y)-\frac{1}{2}
 \big(a^2(x)+a^2(y)+a^2(x-y)\big)}.
\]
%
Moreover, 
as $s:= \|x\| \to \infty$, $\ln \|y\|  \sim \ln s$ and
$\ln \|x-y\|\sim \beta \ln s$ with some $\beta\in [0,1]$, we have
\begin{equation*}
 \IP\big[\{x,y\}\subset \V^\alpha\big]
= s^{-\frac{4\alpha}{4-\beta}+o(1)},
\end{equation*}
and 
polynomially decaying correlations 
(cf.~(\ref{def:cor}) in Remark~\ref{rem24} for the definition), 
\begin{equation}
\label{eq:cor}
\Cor\big(
 {\{x\in \V^\alpha\}}, 
 {\{y\in \V^\alpha\}}\big)
= s^{-\frac{\alpha \beta}{4-\beta}+o(1)}.
\end{equation}
 \item[(v)] Assume that $\ln \|x\|\sim \ln s$, $\ln r\sim\beta \ln s$ with $\beta<1$.
Then, as $s \to \8$,
\begin{equation}
\label{properties_RI_v}
  \IP\big[B(x,r)\subset \V^\alpha\big] 
= s^{-\frac{2\alpha}{2-\beta}+o(1)}.
\end{equation}
\end{itemize}
\end{theo}
These results invite a few comments.
\begin{rem} \label{rem24}
\begin{enumerate}
\item The statement in (i) describes an invariance property 
given that a point is vacant.
We refer to it as the conditional stationarity
\item We can interpret (iii) as follows: 
the conditional law of RI($\alpha$) 
given that a distant site~$x$ is vacant, is similar -- near the origin -- 
to the unconditional law of RI($\alpha/4$). Combined with~(i), the similarity holds near~$x$ as well.
Moreover, one can also estimate the ``local rate'' away
from the origin,
see Figure~\ref{f_change_rate}. More specifically,
observe from Lemma~\ref{l_cap_distantball}~(ii) that
$\capa(A_2)\ll \ln s$ with $s={\rm dist}(0,A_2)$ large implies
$\capa\big(\{0\}\cup A_2\big)=\frac{a(s)}{2}(1+o(1))$.
 If $x$ is at a much larger distance from the origin than~$A_2$, 
say $\ln \|x\| \sim \ln( s^2)$,
then~\eqref{properties_RI_iii} reveals a 
``local rate'' equal to $\frac{2}{7}\alpha$,
that is, 
$\IP[A_2\subset\V^\alpha\mid x \in \V^\alpha]=\exp\big(-\frac{2}{7}
\pi\alpha\capa\big(\{0\}\cup A_2\big)(1+o(1))\big)$;
indeed, the expression in the denominator in~\eqref{properties_RI_iii}
equals approximately $1-\frac{\capa(\{0\}\cup A_2)}{2a(x)}
\approx 1 - \frac{a(s)/2}{2a(s^2)}\approx \frac{7}{8}$.
\item    
Recall that the correlation between two events~$A$ and~$B$, 
\[ 
\label{def:cor}
 \Cor(A,B) = \frac{ {\rm Cov}({\mathbf 1}_A, {\mathbf 1}_ B)}
{[ \Var({\mathbf 1}_A) \Var(  {\mathbf 1}_ B)]^{1/2}}\in [-1,1],
\] 
 can be viewed as the cosine of the angle 
between the vectors ${\mathbf 1}_{A}-\IP(A),
 {\mathbf 1}_B-\IP(B)$ in the corresponding Hilbert space $L^2$. 
Then, 
equation~\eqref{eq:cor} relates the geometry of the lattice 
with the  geometry of the random point process: 
For points~$x$ and~$y$ at large distance~$s$ making a small 
angle $s^{1-\beta}$ with the origin, the random variables 
$\1{x\in \V^\alpha}$ and 
$\1{y\in \V^\alpha}$ make, after centering, 
an angle of order $ s^{-\frac{\alpha \beta}{4-\beta}}$ 
in the space of square integrable random variables.
\item By symmetry, the conclusion of (iv) remains 
the same in the situation when
$\ln \|x\|, \ln \|x-y\|\sim \ln s$ and $\ln \|y\|\sim \beta\ln s$.
\end{enumerate}
\end{rem}
\begin{figure}
\begin{center}
\includegraphics{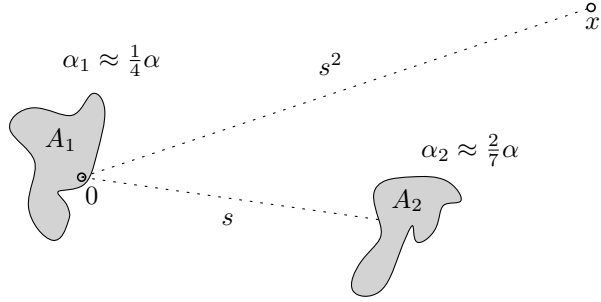}
\caption{How the ``local rate'' looks like if we condition 
on the event that a ``distant'' site is vacant.}
\label{f_change_rate}
\end{center}
\end{figure}
\begin{proof}[Proof of (i) and (ii)]
 To prove~(i), observe that
\[
\capa\big(\{0,x\}\cup A\big) = \capa\big(\{0,x\}\cup (-A+x)\big)
\]
  by symmetry. For the second statement in (i),
note that, for $A'=\{0,x\}\cup A$,
it holds that
$
\capa\big(A'\big) =\capa\big(MA'\big) = \capa\big(\{0,x\}\cup MA \big).
$
Item~(ii) follows from the above mentioned fact that 
$\capa\big(\{0,x\}\big)=\frac{1}{2}a(x)$
together with~\eqref{formula_for_a}.
\end{proof}

We postpone the proof of other parts of Theorem~\ref{t_properties_RI}, 
since it requires some estimates for capacities
of various kinds of sets. 
We now turn to estimates on the cardinality of the vacant set.

\begin{theo}
\label{t_sizevacant}
\begin{itemize}
 \item[(i)] We have
\[
 \E \big(\vert \V^\alpha\cap B(r)\vert \big) \sim
  \begin{cases}
   \frac{2\pi}{2-\alpha}
   e^{-\gamma'\pi\alpha/2} \times r^{2-\alpha}, & \text{ for }\alpha < 2,\\
  2\pi e^{-\gamma'\pi\alpha/2} \times \ln r, & \text{ for }\alpha = 2,\\
   \text{const} , & \text{ for }\alpha > 2.
 \end{cases}
\]
 \item[(ii)] For $\alpha>1$ it holds that $\V^\alpha$ is
 finite a.s. Moreover, $\IP\big[\V^\alpha=\{0\}\big]>0$
 and $\IP\big[\V^\alpha=\{0\}\big]\to 1$ as $\alpha\to\infty$.
 \item[(iii)] For $\alpha \in (0,1)$, we have $|\V^\alpha|=\infty$
a.s. Moreover,
\begin{equation}
\label{eq_emptyball}
 \IP\big[\V^\alpha\cap \big(B(r)\setminus B(r/2)\big)=\emptyset\big]
 \leq r^{-2(1-\sqrt{\alpha})^2+o(1)}.
\end{equation}
\end{itemize}
\end{theo}
It is worth noting that the ``phase transition'' at $\alpha=1$
in~(ii) corresponds to the cover time of the torus, 
as shown in Theorem~\ref{t_conditional} below.

\begin{proof}[Proof of (i) and (ii) (incomplete, in the latter case)]
Part~(i) immediately follows from Theorem~\ref{t_properties_RI}~(ii).

The proof of the part~(ii) is easy in the case $\alpha>2$. 
Indeed, observe first that $\IE|\V^\alpha|<\infty$
implies that $\V^\alpha$ itself is a.s.\ finite.
Also, Theorem~\ref{t_properties_RI}~(ii) actually implies
that $\IE|\V^\alpha \setminus \{0\}|\to 0$ as $\alpha\to\infty$,
so $\IP\big[\V^\alpha=\{0\}\big]\to 1$ by the Chebyshev inequality.

Now, let us prove that, in general, $\IP\big[|\V^\alpha|<\infty\big]=1$
implies that $\IP\big[\V^\alpha=\{0\}\big]>0$.
Indeed, if~$\V^\alpha$ is a.s.\ finite, then one can 
find a sufficiently large~$R$ such that 
$\IP\big[|\V^\alpha\cap (\Z^2\setminus B(R))|=0\big]>0$.
Since $\IP[x\notin \V^\alpha]>0$ for any $x\neq 0$,
the claim that $\IP\big[\V^\alpha=\{0\}\big]>0$ 
follows from the FKG inequality applied to events
$\{x\notin \V^\alpha\}$, $x\in B(R)$ together
with $\big\{|\V^\alpha\cap (\Z^2\setminus B(R))|=0\big\}$.
\end{proof}

As before, we postpone the proof of part~(iii) and 
the rest of part~(ii) of
Theorem~\ref{t_sizevacant}. 
Let us remark that we believe that the right-hand side
of~\eqref{eq_emptyball} gives the correct order of decay 
of the above probability; we, however, do not have a
rigorous argument at the moment.
Also, note that the question
whether $\V^1$ is a.s.\ finite or not, is open. 

Let us now give
a heuristic explanation about the unusual
behaviour of the model for $\alpha \in (1,2)$:
in this non-trivial interval, the vacant set is 
a.s.\ finite but its expected size is infinite.
The reason is the following:
the number of $\s$-walks that hit~$B(r)$ has Poisson law
with rate of order~$\ln r$ (recall~\eqref{capa_ball}).
 Thus, decreasing this number by a 
constant factor (with respect to the expectation) has only
a polynomial cost. On the other hand, by doing so, we increase 
the probability that a site~$x\in B(r)$ is vacant for all $x\in B(r)$
at once, which increases the expected size of $\V^\alpha\cap B(r)$
by a polynomial factor. It turns out that this effect causes
the actual number of uncovered sites in~$B(r)$ to be typically
of much smaller order then the expected number of uncovered sites 
there.

\subsection{Simple random walk on a discrete torus and 
its relationship with random interlacements}
\label{s_rw_interl}
Now, we state our results for random walk on the torus. 
Let~$(X_k, k\geq 0)$ be simple random walk on~$\Z^2_n$
with~$X_0$ chosen uniformly at random. 
 Define the entrance time
to the site $x\in \Z^2_n$ by
\begin{equation}
\label{eq_defTx}
T_n(x) = \inf\{t\geq 0: X_t=x\},
\end{equation}
and the \emph{cover time} of the torus by
\begin{equation}
\label{eq_defT}
\T_n = \max_{x\in \Z^2_n} T_n(x).
\end{equation}
Let us also define the \emph{uncovered set} at time~$t$,
\begin{equation}
\label{df_U_t}
 U_t^{(n)} = \{x\in\Z^2_n : T_n(x) > t\}.
\end{equation}

Denote by $\Upsilon_n:\Z^2 \to \Z^2_n$,
$\Upsilon_n(x,y)=(x \mod n, y\mod n)$, 
the natural projection modulo~$n$. 
Then, if~$S_0$ were chosen uniformly at random 
on any fixed $n\times n$ square,
 we can write $X_k=\Upsilon_n (S_k)$.
Similarly,
$B(y,r)\subset \Z^2_n$ is defined by $B(y,r)=\Upsilon_n B(z,r)$,
where $z\in\Z^2$ is such that $\Upsilon_n z = y$. 
Let also
\[
t_\alpha:=\frac{4\alpha}{\pi}n^2\ln^2 n
\]
(recall that, as mentioned in Section~\ref{s_late},
 $\alpha=1$ corresponds to the leading-order term of the 
expected cover time of the torus).
In the following
theorem, we prove that, given that~$0$ is uncovered, the law
of the uncovered set around~$0$ at time~$t_\alpha$
is close to that of RI($\alpha$):
\begin{theo}
\label{t_conditional} 
Let $\alpha>0$ and~$A$ be a finite subset of~$\Z^2$ 
such that $0\in A$. Then, we have 
\begin{equation}
\label{eq_conditional}
 \lim_{n\to\infty}\IP[\Upsilon_n A \subset 
   U_{t_\alpha}^{(n)} \mid 0\in U_{t_\alpha}^{(n)}]
  = \exp\big(-\pi\alpha\capa(A)\big).
\end{equation}
\end{theo}

Let us also mention that, for higher dimensions, results
similar to the above theorem have appeared in the literature,
see Theorem~0.1 of~\cite{Szn09-2} and Theorem~1.1 of~\cite{W08}.
Also, stronger ``coupling'' results are available, see
e.g.~\cite{B13,CT15,Szn09-1,TW11}.

The proof of this theorem will be presented in Section~\ref{s_proofs}.
\begin{figure}
\begin{center}
\includegraphics{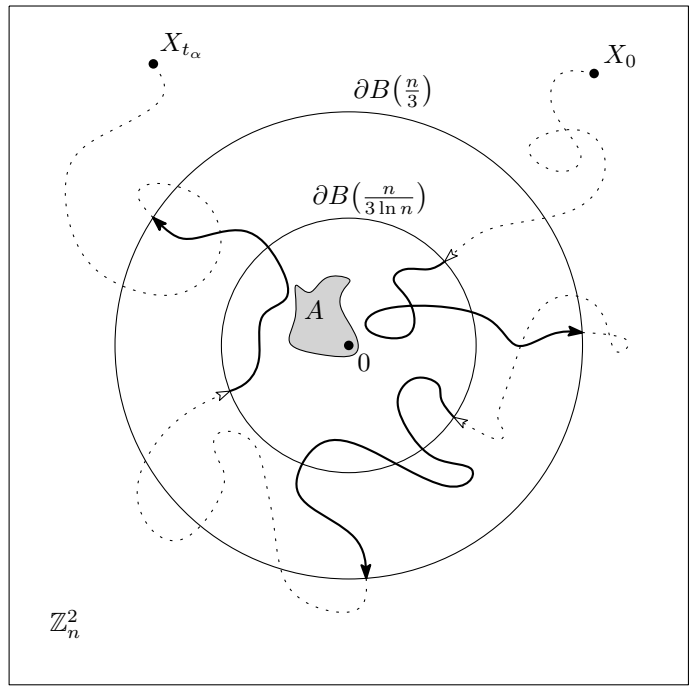}
\caption{Excursions (depicted as the solid
pieces of the trajectory) of the SRW on the torus~$\Z^2_n$}
\label{f_excursions}
\end{center}
\end{figure}
To give a brief heuristic explanation for~\eqref{eq_conditional},
 consider the \emph{excursions} of the random walk~$X$
between $\partial B\big(\frac{n}{3\ln n}\big)$ and $\partial B(n/3)$
up to time~$t_\alpha$.
We hope that Figure~\ref{f_excursions} is self-explanatory;
formal definitions are presented in Section~\ref{s_aux_torus}.
It is possible 
to prove that the number of these excursions
 is concentrated around 
$N=\frac{2\alpha\ln^2 n}{\ln\ln n}$. Also, 
one can show that each excursion hits a finite set~$A$ (such that
$0\in A$)
approximately independently of the others (also when
conditioning on $0\in U_{t_\alpha}^{(n)}$), with 
probability roughly equal to $p=\frac{\pi\ln\ln n}{2\ln^2 n}\capa(A)$.
This heuristically gives~\eqref{eq_conditional} since 
$(1-p)^{N}\approx \exp\big(-\pi\alpha\capa(A)\big)$.

\section{Some auxiliary facts and estimates}
\label{s_aux}
In this section we collect lemmas of all sorts
that will be needed in the sequel. 

\subsection{Simple random walk in~$\Z^2$}
\label{s_aux_SRW}
First, we recall a couple of basic facts for the exit
probabilities of simple random walk.
\begin{lem}
\label{l_exit_balls}
For all $x, y \in \Z^2$ and $R>0$ with $x \in B(y,R), \|y\|\leq R-2$, we have
\begin{equation}
\label{nothit_0}
 P_x\big[\tau_1(0)>\tau_1\big(\partial B(y,R)\big)\big] 
   = \frac{a(x)}{a(R) + O\big(\frac{\|y\|\vee 1}{R}\big)}\;, 
   \end{equation}
 and for all
 $y \in B(r), x \in B(y,R)\setminus B(r)$ with $  r+\|y\| \leq R-2$, we have
\begin{equation}
 \label{nothit_r}
P_x\big[\tau_1(\partial B(r))>\tau_1\big(\partial B(y,R)\big)\big] = 
\frac{a(x)-a(r)+O(r^{-1})}{a(R)-a(r)
+O\big(\frac{\|y\|\vee 1}{R}+r^{-1}\big)}\;,
\end{equation}
as $r,R\to \infty$. 
\end{lem}
\begin{proof}
 Both \eqref{nothit_0} and \eqref{nothit_r} 
 are easily deduced from the following argument:
recall that the sequence  $a(S_{k\wedge \tau_0(0)}), k \geq 0, $ is a martingale,
and apply the Optional Stopping Theorem 
together with~\eqref{formula_for_a} and~\eqref{real_a}. For the second statement,
with  
\[
\tau= \tau_1(\partial B(r)) \wedge \tau_1\big(\partial B(y,R)\big)
\quad \text{ and } \quad q=P_x\big[\tau_1(\partial B(r))>\tau_1
\big(\partial B(y,R)\big)\big],
\]
 we write 
\begin{align*}
a(x)&= E_x\big[ a(S_{ \tau}); \tau = \tau_1(\partial B(y,R))\big]+ 
E_x\big[ a(S_{ \tau}); \tau = \tau_1(\partial B(r))\big] \\
&= q \big(a(R)+ O\big(\textstyle\frac{\|y\|\vee 1}{R}\big)\big)
 + (1-q) \big(a(r)+O(r^{-1})\big),
\end{align*}
yielding \eqref{nothit_r}. The first statement has a similar proof. 
\end{proof}

We have an estimate for more general sets.
\begin{lem}
\label{l_hit_A}
 Let~$A$ be a finite subset of~$\Z^2$ such that
$A\subset B(r)$ for some $r > 0$. We have for
$r+1\leq\|x\|\leq R-2$, 
$\|x\|+\|y\|\leq R-1$
\begin{equation}
\label{nothit_A}
 P_x\big[\tau_1(A)>\tau_1\big(\partial B(y,R)\big)\big] 
   = \frac{a(x)-\capa(A)+O\big(\frac{r\ln r \ln \|x\|}{\|x\|}\big)}
{a(R)-\capa(A) 
+ O\big(\frac{\|y\|\vee 1}{R}+\frac{r\ln r \ln \|x\|}{\|x\|}\big)}.
\end{equation}
\end{lem}

\begin{proof}
First, observe that, for $u \in A$, 
\begin{equation}
\label{first_observe}
 P_x\big[ S_{\tau_1(A)}=u \big]
 = \hm_A(u)\big(1+O\big(\textstyle\frac{r\ln  \|x\|}{\|x\|}\big)\big).
\end{equation}
Indeed, from~(6.42) and the last display on page~178 of~\cite{LL10}
we obtain
\[
 \frac{P_x\big[ S_{\tau_1(A)}=u \big]}{P_{x'}\big[ S_{\tau_1(A)}=u \big]}
= 1+O\big(\textstyle\frac{r\ln  \|x\|}{\|x\|}\big)
\]
for all $x'\in \partial B(\|x\|)$. Since~$\hm_A(u)$ is a linear combination
of $P_{x'}\big[ S_{\tau_1(A)}=u \big]$, $x'\in \partial B(\|x\|)$,
\eqref{first_observe} follows.

Then, we can write
\begin{align*}
\lefteqn{P_x\big[ S_{\tau_1(A)}= u , \tau_1(A) <
 \tau_1\big(\partial B(y,R)\big)\big]}\\
 &= P_x\big[ S_{\tau_1(A)} = u \big] - P_x\big[ S_{\tau_1(A)} = u ,
 \tau_1(A)>\tau_1\big(\partial B(y,R)\big)\big]\\
& = \hm_A(u)\big(1+O\big(\textstyle\frac{r\ln \|x\|}{\|x\|}\big)\big) - P_x\big[\tau_1(A) > \tau_1\big(\partial B(y,R)\big)\big] \\
&\qquad  
 \times \sum_{z\in \partial B(y,R)}P_z\big[ S_{\tau_1(A)} 
= u \big]P_x\big[ S_{\tau_1(\partial B(y,R))} = z \mid \tau_1(A) 
 > \tau_1\big(\partial B(y,R)\big)\big] \\
&= \hm_A(u)\big(1+O\big(\textstyle\frac{r\ln \|x\|}{\|x\|}\big)\big) - P_x\big[\tau_1(A) > \tau_1\big(\partial B(y,R)\big)\big]
\hm_A(u)\big(1+O\big(\textstyle\frac{r\ln R}{R}\big)\big) \\
&= \hm_A(u)P_x\big[\tau_1(A) < \tau_1\big(\partial B(y,R)\big)\big] \big(1+O\big(\textstyle\frac{r\ln \|x\|}{\|x\|}\big)\big),
\end{align*}
using twice the above observation and also that 
$B(\|x\|)\subset B(y,R)$, yielding 
\begin{equation}
\label{enter_prob}
P_x\big[ S_{\tau_1(A)}=u \mid \tau_1(A)<\tau_1\big(\partial B(y,R)\big)\big]
 = \hm_A(u)\big(1+O\big(\textstyle\frac{r\ln \|x\|}{\|x\|}\big)\big).
\end{equation}
 Abbreviating $q:=P_x\big[\tau_1(A)>\tau_1\big(\partial B(y,R)\big)\big]$,
we have by the Optional Stopping Theorem and~\eqref{enter_prob}
\begin{align*}
a(x) &= q\big(a(R) + O\big(\textstyle\frac{\|y\|\vee 1}{R}\big)\big)
 + (1-q)\sum_{u\in A\setminus\{0\}}
a(u)\hm_A(u) \big(1+O\big(\textstyle\frac{r\ln \|x\|}{\|x\|}\big)\big)\\
 &= q\big(a(R) + O\big(\textstyle\frac{\|y\|\vee 1}{R}\big)\big)
 + (1-q)\big(1+O\big(\textstyle\frac{r\ln \|x\|}
{\|x\|}\big)\big)\capa(A),
\end{align*}
and~\eqref{nothit_A} follows (observe also that 
$\capa(A)\leq \capa(B(r))=\frac{2}{\pi}\ln r + O(1)$, see~\eqref{capa_ball}).
\end{proof}

\subsection{Simple random walk conditioned on not 
hitting the origin}
\label{s_aux_hat_s}
Next, we relate the probabilities of certain events
for the walks~$S$ and~$\s$. For $M \subset \Z^2$, let~$\Gamma^{(x)}_M$ be the set of all 
nearest-neighbour finite
trajectories that start at~$x\in M\setminus\{0\}$ 
and end when entering~$\partial M$ for the first time;
 denote also~$\Gamma^{(x)}_{y,R}=\Gamma^{(x)}_{B(y,R)}$.
 For~$A\subset \Gamma^{(x)}_M$ write
 $S\in A$ if there exists~$k$ such that 
$(S_0,\ldots,S_k)\in A$ (and the same for
the conditional walk~$\s$). In the next
result we show that $P_x\big[\;\cdot\mid\tau_1(0)>
\tau_1\big(\partial B(R)\big)\big]$ 
and $\hP_x[\,\cdot\,]$ are 
almost indistinguishable on~$\Gamma^{(x)}_{0,R}$ (that is, the conditional
law of~$S$ almost coincides with the unconditional law 
of~$\s$). A similar result holds for excursions on a 
``distant'' (from the origin) set.
\begin{lem}
\label{l_relation_S_hatS}
\begin{itemize}
 \item[(i)] Assume $A\subset \Gamma^{(x)}_{0,R}$.
We have 
\begin{equation}
\label{eq_relation_S_hatS}
P_x\big[S\in A\mid \tau_1(0)>\tau_1\big(\partial B(R)\big)\big]
 =\hP_x\big[\s \in A\big] \big(1+O((R \ln R)^{-1})\big).  
\end{equation}
\item[(ii)] Assume that $A\subset \Gamma^{(x)}_M$
and suppose that $0\notin M$, and denote $s=\dist(0,M)$,
$r=\diam(M)$. Then, for $x \in M$,
\begin{equation}
\label{eq_relation_S_hatS2}
P_x[S\in A]
 =\hP_x[\s \in A]\Big(1+O\Big(\frac{r}{s\ln s}\Big)\Big).
\end{equation}
\end{itemize}
\end{lem}

\begin{proof}
Let us prove part~(i).
Assume without loss of generality that no trajectory from~$A$
passes through the origin.
 Then, it holds that 
\[
 \hP_x[\s \in A] = \sum_{\vr\in A}
  \frac{a(\vr_{\text{end}})}{a(x)}\Big(\frac{1}{4}\Big)^{|\vr|},
\]
with $|\vr|$ the length of $\vr$.
On the other hand, by~\eqref{nothit_0}
\[
P_x[S\in A\mid  \tau_1(0)>\tau_1\big(\partial B(R)\big)]
 = \frac{a(R) + O(R^{-1})}{a(x)}
  \sum_{\vr\in A}
  \Big(\frac{1}{4}\Big)^{|\vr|}.
\]
Since
$\vr_{\text{end}}\in\partial B(R)$, we have 
$a(\vr_{\text{end}})=a(R) + O(R^{-1})$,
and so~\eqref{eq_relation_S_hatS} follows.

The proof of part~(ii) is analogous (observe that 
$a(y_1)/a(y_2)=1+O\big(\frac{r}{s\ln s}\big)$ for any $y_1, y_2\in M$).
\end{proof}

As observed in Section~\ref{s_RI}, the random walk~$\s$ 
is transient.
Next,
we estimate the 
probability that the $\s$-walk avoids a ball centered
at the origin:
\begin{lem}
\label{l_escape_hatS}
Assume $r \geq 1$ and $\|x\| \geq r+1$. We have
\[
\hP_x\big[\htau_1\big(B(r)\big)=\infty\big] = 
1-\frac{a(r)+O(r^{-1})}{a(x)}.
\]
\end{lem}

\begin{proof}
By Lemma~\ref{l_relation_S_hatS}~(i) we have
\[
 \hP_x\big[\htau_1\big(B(r)\big)=\infty\big] = 
  \lim_{R\to \infty}
  P_x\big[\tau_1(\partial B(r))>\tau_1\big(\partial B(R)\big)
   \mid \tau_1(0)>\tau_1\big(\partial B(R)\big)\big].
\]
The claim then follows from~\eqref{nothit_0}--\eqref{nothit_r}.
\end{proof}

\begin{rem}
Alternatively, one can deduce the proof 
of Lemma~\ref{l_escape_hatS} from the fact 
that $1/a(\s_{k\wedge \htau_0(\mathcal{N})})$ is a martingale,
together with the Optional Stopping Theorem.
\end{rem} 

We will need also an expression for the probability
of avoiding \emph{any} finite set containing the origin:
\begin{lem}
\label{l_avoid_A}
 Assume that $0\in A \subset B(r)$, and $\|x\|\geq r+1$.
Then
\begin{equation}
\label{eq_avoid_A}
\hP_x[\htau_1(A)=\infty] = 1-\frac{\capa(A)}{a(x)}
+O\Big(\frac{r\ln r \ln\|x\|}{\|x\|}\Big).
\end{equation}
\end{lem}

\begin{proof}
 Indeed, using Lemmas~\ref{l_hit_A} and~\ref{l_relation_S_hatS}~(i)
together with~\eqref{nothit_0},
we write
\begin{align*}
 \hP_x[\htau_1(A)=\infty] &= \lim_{R\to\infty}
 P_x\big[\tau_1(A)>\tau_1\big(\partial B(R)\big)
\mid \tau_1(0)>\tau_1\big(\partial B(R)\big)\big] \\
&=\lim_{R\to\infty} \frac{a(R)+O(R^{-1})}{a(x)}
\times \frac{a(x)-\capa(A)+O\big(\frac{r\ln r \ln \|x\|}{\|x\|}\big)}
{a(R)-\capa(A) + O\big(R^{-1}+\frac{r\ln r \ln \|x\|}{\|x\|}\big)},
\end{align*}
thus obtaining~\eqref{eq_avoid_A}.
\end{proof}

It is also possible to obtain 
exact expressions for one-site escape probabilities,
and probabilities of (not) hitting a given site:
\begin{align}
 \hP_x[\htau_1(y)<\infty] &= \frac{a(x)+a(y)-a(x-y)}{2a(x)},
\label{escape_from_site1}\\
\intertext{for $x\neq y$, $x,y\neq 0$ and }
 \hP_x[\htau_1(x)<\infty] &= 1-\frac{1}{2a(x)}
\label{escape_from_site2}
\end{align}
for $x\neq 0$. We temporarily postpone the proof of
 \eqref{escape_from_site1}--\eqref{escape_from_site2}. 
Observe that, in particular, 
we recover from~\eqref{escape_from_site2} the transience of~$\s$.
Also, observe that~\eqref{escape_from_site1} implies 
the following surprising fact:
 for any~$x\neq 0$, 
\[
 \lim_{y\to\infty} \hP_x[\htau_1(y)<\infty] = \frac{1}{2}.
\]
The above relation leads to the following heuristic explanation for
Theorem~\ref{t_properties_RI}~(iii)
(in the case when~$A$ is fixed and $\|x\|\to\8$). Since the probability
of hitting a distant site is about~$1/2$, by conditioning 
that this distant site is vacant, we essentially throw away
three quarters of the trajectories that pass through 
a neighbourhood of the origin: indeed, the double-infinite
trajectory has to avoid this distant site two times, 
before and after reaching that neighbourhood.

Let us state several other
general estimates, for the probability
of (not) hitting a given set (which is, typically,
far away from the origin), 
or, more specifically, a disk:
\begin{lem}
\label{l_escape_from_ball}
Assume that $x\notin B(y,r)$ and $\|y\|>2r\geq 1$. 
Abbreviate also $\Psi_1=\|y\|^{-1}r$,
$\Psi_2=\frac{r\ln r\ln\|y\|}{\|y\|}$,
$\Psi_3=r\ln r\big(\frac{\ln\|x-y\|}{\|x-y\|}
+\frac{\ln\|y\|}{\|y\|}\big)$.
\begin{itemize}
\item[(i)] We have 
\begin{equation}
\label{eq_escape_from_ball}
 \hP_x\big[\htau_1(B(y,r))<\infty\big]
= \frac{\big(a(y)+O(\Psi_1)\big)\big(a(y)+a(x)-a(x-y)
+O(r^{-1})\big)}{a(x)\big(2a(y)-a(r)+O(r^{-1}+\Psi_1)\big)}.
\end{equation}
\item[(ii)] Consider now any nonempty set $A\subset B(y,r)$.
Then, it holds that
\begin{equation}
\label{eq_escape_from_anyset}
 \hP_x\big[\htau_1(A)<\infty\big]
= \frac{\big(a(y)+O(\Psi_1)\big)\big(a(y)+a(x)-a(x-y)
+O(r^{-1}+\Psi_3
)\big)}
{a(x)\big(2a(y)-\capa(A)+ O(\Psi_2)
\big)}.
\end{equation}
\end{itemize}
\end{lem}
Observe that~\eqref{eq_escape_from_ball} is \emph{not}
a particular case of~\eqref{eq_escape_from_anyset};
this is because~\eqref{nothit_r}
typically provides a more precise estimate than~\eqref{nothit_A}.
\begin{proof}
 Fix a (large) $R>0$, such that $R>\max\{\|x\|,\|y\|+r\}+1$. Denote 
\begin{align*}
 h_1 &= P_x\big[\tau_1(0)<\tau_1\big(\partial B(R)\big)\big],\\
 h_2 &= P_x\big[\tau_1\big(B(y,r)\big)<\tau_1\big(\partial B(R)\big)\big],\\
 p_1 &= P_x\big[\tau_1(0)<\tau_1\big(\partial B(R)\big)\wedge
 \tau_1(B(y,r))\big],\\
 p_2 &= P_x\big[\tau_1\big(B(y,r)\big)<\tau_1\big(\partial B(R)\big)\wedge
 \tau_1(0)\big],\\
 q_{12} &= P_0\big[\tau_1\big(B(y,r)\big)
<\tau_1\big(\partial B(R)\big)\big],\\
 q_{21} &= P_\nu\big[\tau_1(0)<\tau_1\big(\partial B(R)\big)\big],
\end{align*}
where~$\nu$ is the entrance measure to~$B(y,r)$ starting from $x$
conditioned on the event 
$\big\{\tau_1\big(B(y,r)\big)
<\tau_1\big(\partial B(R)\big)\wedge \tau_1(0)\big\}$,
see Figure~\ref{f_p12}. 
\begin{figure}
\begin{center}
\includegraphics[width=0.64\textwidth]{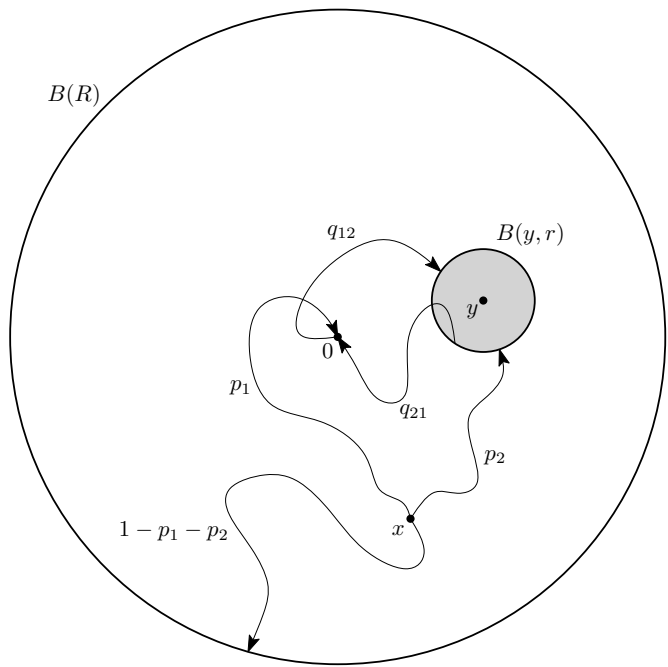}
\caption{On the proof of Lemma~\ref{l_escape_from_ball}}
\label{f_p12}
\end{center}
\end{figure}
Using Lemma~\ref{l_exit_balls}, we obtain
\begin{align}
 h_1 &= 1-\frac{a(x)}{a(R)+O(R^{-1})},
\label{expr_h1}\\
 h_2 &= 1-\frac{a(x-y)-a(r)+O(r^{-1})}{a(R)-a(r)+O(R^{-1}\|y\|+r^{-1})},
\label{expr_h2}
\end{align}
and 
\begin{align}
q_{12} &= 1-\frac{a(y)-a(r)+O(r^{-1})}{a(R)-a(r)+O(R^{-1}\|y\|+r^{-1})},
\label{expr_q12} \\
q_{21} &=1-\frac{a(y)+O(\|y\|^{-1}r)}{a(R)+O(R^{-1}\|y\|)}.
\label{expr_q21} 
\end{align}

Then, as a general fact, it holds that
\begin{align*}
 h_1 &= p_1 + p_2q_{21},\\
 h_2 &= p_2 + p_1q_{12}.
\end{align*}
Solving these equations with respect to $p_1,p_2$, we
obtain
\begin{align}
 p_1 &= \frac{h_1-h_2q_{21}}{1-q_{12}q_{21}},
\label{expr_p1}\\
 p_2 &= \frac{h_2-h_1q_{12}}{1-q_{12}q_{21}},
\label{expr_p2}
\end{align}
and so, using \eqref{expr_h1}--\eqref{expr_q21}, we write
\begin{align*}
 \lefteqn{P_x\big[\tau_1(B(y,r))<\tau_1\big(\partial B(R)\big) 
\mid \tau_1(0)>\tau_1\big(\partial B(R)\big)\big]}\\
  & = \frac{p_2(1-q_{21})}{1-h_1}
 \\
  & = \frac{(h_2 - h_1q_{12})(1-q_{21})}{(1-h_1)(1-q_{12}q_{21})}\\
&= \Bigg(\frac{a(x)}{a(R)+O(R^{-1})}
 +\frac{a(y)-a(r)+O(r^{-1})}{a(R)-a(r)+O(R^{-1}\|y\|+r^{-1})}\\
&\qquad
 - \frac{a(x-y)-a(r)+O(r^{-1})}{a(R)-a(r)+O(R^{-1}\|y\|+r^{-1})}
+O\Big(\frac{\ln\|x-y\|\ln\|y\|}{\ln^2 R}\Big)\Bigg)\\
&\quad \times \frac{a(y)+O(\|y\|^{-1}r)}{a(R)+O(R^{-1}\|y\|)}\times
\Big(\frac{a(x)}{a(R)+O(R^{-1})}\Big)^{-1}
\\
&\quad \times \Bigg(\frac{a(y)-a(r)+O(r^{-1})}
{a(R)-a(r)+O(R^{-1}\|y\|+r^{-1})} 
+ \frac{a(y)+O(\|y\|^{-1}r)}{a(R)+O(R^{-1}\|y\|)}
+O\Big(\frac{\ln^2\|y\|}{\ln^2 R}\Big)\Bigg)^{-1}.
\end{align*}
Sending~$R$ to infinity,
we obtain the proof of~\eqref{eq_escape_from_ball}.

To prove~\eqref{eq_escape_from_anyset}, we use the same procedure.
Define $h'_{1,2}$, $p'_{1,2}$, $q'_{12}$, $q'_{21}$ in the same
way but with~$A$ in place of~$B(y,r)$.
It holds that $h'_1=h_1$, $q'_{21}$ is 
expressed in the same way as~$q_{21}$
(although $q'_{21}$ and $q_{21}$ are not necessarily equal, 
the difference
is only in the error terms $O(\cdot)$)
and, by Lemma~\ref{l_hit_A},
\begin{align*}
 h'_2 &= 1-\frac{a(x-y)-\capa(A)
+O\big(\frac{r\ln r\ln\|x-y\|}{\|x-y\|}\big)}
{a(R)-\capa(A)+O\big(R^{-1}\|x-y\|
+\frac{r\ln r\ln\|x-y\|}{\|x-y\|}\big)},\\
q'_{12} &= 1-\frac{a(y)-\capa(A)
+O\big(\frac{r\ln r\ln\|y\|}{\|y\|}\big)}
{a(R)-\capa(A)+O\big(R^{-1}\|y\|
+\frac{r\ln r\ln\|y\|}{\|y\|}\big)}.
\end{align*}
After the analogous calculations, we obtain~\eqref{eq_escape_from_anyset}.
\end{proof}
\begin{proof}[Proof of relations \eqref{escape_from_site1}--\eqref{escape_from_site2}]
Formula~\eqref{escape_from_site2} rephrases~\eqref{escape_identity} with $A=\{0,x\}$.
 Identity~\eqref{escape_from_site1} follows from the same proof as in
 Lemma~\ref{l_escape_from_ball} (i), using~\eqref{nothit_0}
instead of~\eqref{nothit_r}.
\end{proof}

\subsection{Harmonic measure and capacities}
\label{s_aux_capacities}
Next, we need a formula for calculating the capacity
of three-point sets:
\begin{lem}
\label{l_cap_triangle}
Let $x_1,x_2,x_3\in\Z^2$, and abbreviate $v_1=x_2-x_1$, $v_2=x_3-x_2$,
$v_3=x_1-x_3$. Then, the capacity of the set $A=\{x_1,x_2,x_3\}$
is given by the formula
\begin{equation}
\label{eq_cap_triangle}
 \frac{a(v_1)a(v_2)a(v_3)}
 {a(v_1)a(v_2)+a(v_1)a(v_3)+a(v_2)a(v_3)-\frac{1}{2}
 \big(a^2(v_1)+a^2(v_2)+a^2(v_3)\big)}.
\end{equation}
\end{lem}

\begin{proof}
By Proposition~6.6.3 and Lemma~6.6.4 of~\cite{LL10},
the inverse capacity of~$A$ is equal to the sum of entries
of the matrix
\[
 a_A^{-1} = 
 \begin{pmatrix}
  0 & a(v_1) & a(v_3)\\
  a(v_1) & 0 & a(v_2)\\
  a(v_3) & a(v_2) & 0
 \end{pmatrix}^{-1}
 =\frac{1}{2}
 \begin{pmatrix}
  -\frac{a(v_2)}{a(v_1)a(v_3)} & \frac{1}{a(v_1)} & 
  \frac{1}{a(v_3)}\\
  \frac{1}{a(v_1)} & - \frac{a(v_3)}{a(v_1)a(v_2)} & 
  \frac{1}{a(v_2)}\\
  \frac{1}{a(v_3)} & \frac{1}{a(v_2)}
  & -\frac{a(v_1)}{a(v_2)a(v_3)}
 \end{pmatrix},
\]
and this implies~\eqref{eq_cap_triangle}.
\end{proof}

Before proceeding, let us notice the following immediate consequence
of Lemma~\ref{l_avoid_A}: for any finite $A\subset\Z^2$
such that $0\in A$, we have
\begin{equation}
\label{expr_Cap}
 \capa(A) = \lim_{\|x\|\to\infty} a(x)\hP_x[\htau_1(A)<\infty].
\end{equation}

Next, we need estimates for the $\s$-capacity
of a ``distant'' set, and, in particular
of a ball which does not contain the origin.
Recall the notations $\Psi_{1,2,3}$ from Lemma~\ref{l_escape_from_ball}.
\begin{lem}
\label{l_cap_distantball}
  Assume that $\|y\|>2r\geq 1$. 
\begin{itemize}
 \item[(i)]
We have
\begin{equation}
\label{eq_cap_distball}
 \capa\big(\{0\}\cup B(y,r)\big)
  = \frac{\big(a(y)+O(\Psi_1)\big)\big(a(y)+O(r^{-1})\big)}
{2a(y)-a(r)+O(r^{-1}+\Psi_1)}.
\end{equation}
 \item[(ii)] Suppose that $A\subset B(y,r)$. Then
\begin{equation}
\label{eq_cap_distantset}
\capa\big(\{0\}\cup A\big) =\frac{\big(a(y)+O(\Psi_1)\big)
\big(a(y)+O(r^{-1}+\Psi_2 )\big)}
{2a(y)-\capa(A)+ O(\Psi_2)}.
\end{equation}
\end{itemize}
\end{lem}

\begin{proof}
This immediately follows from~\eqref{expr_Cap} 
and Lemma~\ref{l_escape_from_ball} (observe that
$a(x)-a(x-y)\to 0$ as $x\to \infty$ and~$\Psi_3$ becomes~$\Psi_2$).
\end{proof}

We also need to compare the harmonic measure on a set 
(distant from the origin) to
the entrance measure of the $\s$-walk started far away from that set.
\begin{lem}
\label{l_entrance_hat_s}
Assume that $A$ is a finite subset of $\Z^2$, 
$0\notin A$, $x\neq 0$, and also that
$2\,\diam(A)<\dist(x,A) < \frac{1}{4} \dist(0,A)$.
Abbreviate $u=\diam(A)$, $s=\dist(x,A)$. 
Assume also that $A'\subset \Z^2$ (finite, infinite, or even possibly
empty) is such that $\dist(A,A')\geq s+1$ (for definiteness,
we adopt the convention $\dist(A,\emptyset)=\infty$ for any~$A$).
Then, for $y \in A$, it holds that
\begin{equation}
\label{eq_entrance_hat_s}
\hP_x\big[\s_{\htau_1(A)}=y\mid \htau_1(A)<\infty,
   \htau_1(A)<\htau_1(A')\big] = \hm_A(y)\Big(1+
  O\Big(
\frac{u\ln s}{s}\Big)\Big).
\end{equation}
\end{lem}

\begin{proof}
Let $z_0\in A$ be such that $\|z_0-x\|=s$, and observe that
$A'\cap B(z_0,s) = \emptyset$.
Define the discrete circle $L=\partial B(z_0,s)$; 
observe that $x\in L$ and 
$\dist(z',A)\geq s/2$ for all $z'\in L$.
 Let
\[
 \sigma = \sup\big\{0\leq k\leq \htau_1(A) : \s_k\in L\big\}
\]
be the \emph{last} time before $\htau_1(A)$ when the trajectory
passes through~$L$. Note also that for all~$z\in L$
(recall~\eqref{enter_prob})
\begin{equation}
\label{harmonic_to_A}
 P_z\big[S_{\tau_1(A)}=y \mid \tau_1(A)<\tau_1(L)\big] 
   = \hm_A(y)\Big(1+O\Big(\frac{u\ln s}{s}\Big)\Big).
\end{equation}

 Using the Markov property of~$\s$, we write
\begin{align}
\lefteqn{ \hP_x\big[\htau_1(A)<\infty,
  \htau_1(A)<\htau_1(A'),\s_{\htau_1(A)}=y\big] } \nonumber\\
 &= \sum_{k\geq 0, z\in L} \hP_x\big[\htau_1(A)<\infty, 
 \htau_1(A)<\htau_1(A'),\sigma=k,
 \s_\sigma=z, \s_{\htau_1(A)}=y\big]
 \nonumber\\
 &=\sum_{k\geq 0, z\in L} \hP_x\big[\s_k=z, \s_\ell\notin A\cup A'
 \text{ for all }\ell\leq k\big]
 \nonumber\\
 & \qquad\quad\quad \times
 \hP_z\big[ \htau_1(A)<\infty, \htau_1(A)<\htau_1(A'),
  \s_{\htau_1(A)}=y,\s_\ell\notin L \text{ for all }
 \ell\leq \htau_1(A)\big].
 \label{conta_sigma}
\end{align}
Abbreviate $r=\dist(0,A)$.
Now, observe that the last term in~\eqref{conta_sigma} only involves 
trajectories that lie in $B(z_0,s)$,
and we have $\dist\big(0,B(z_0,s)\big)\geq r/2$. 
 So, we can use Lemma~\ref{l_relation_S_hatS}~(ii) together 
 with~\eqref{harmonic_to_A} to write
\begin{align*}
 \lefteqn{\hP_z\big[ \htau_1(A)<\infty, \htau_1(A)<\htau_1(A'),
   \s_{\htau_1(A)}=y,\s_\ell\notin L \text{ for all }
 \ell\leq \htau_1(A)\big] } \\
 &= P_z\big[S_{\tau_1(A)}=y,S_\ell\notin L \text{ for all }
 \ell\leq \tau_1(A)\big]\Big(1+O\Big(\frac{u}{r\ln r}\Big)\Big)\\
 &= P_z\big[S_{\tau_1(A)}=y \mid \tau_1(A)<\tau_1(L)\big]
 P_z[\tau_1(A)<\tau_1(L)]\Big(1+O\Big(\frac{u}{r\ln r}\Big)\Big)\\
 &= \hm_A(y)\hP_z\big[\htau_1(A)<\htau_1(L)\big]\Big(1+
  O\Big(\frac{u}{r\ln r}+\frac{u\ln s}{s}\Big)\Big).
\end{align*}
Inserting this back to~\eqref{conta_sigma}, we obtain
\begin{align*}
\lefteqn{ \hP_x\big[\htau_1(A)<\infty,\htau_1(A)<\htau_1(A'),
 \s_{\htau_1(A)}=y\big] } \\
 &=\hm_A(y) \Big(1+
  O\Big(\frac{u}{r\ln r}+\frac{u\ln u}{s}\Big)\Big)\\
& \qquad \times  \sum_{k\geq 0, z\in L} \hP_x\big[\s_k=z, 
 \s_\ell\notin A\cup A'
 \text{ for all }\ell\leq k\big]\hP_z\big[\htau_1(A)<\htau_1(L)\big]\\
 &= \hm_A(y) \hP_z\big[\htau_1(A)<\infty,\htau_1(A)<\htau_1(A')\big]
\Big(1+O\Big(\frac{u}{r\ln r}+\frac{u\ln s}{s}\Big)\Big),
\end{align*}
and this concludes the proof of Lemma~\ref{l_entrance_hat_s}
(observe that the first term in~$O(\cdot)$ is of smaller order
than the second one).
\end{proof}

\subsection{Random walk on the torus and its excursions}
\label{s_aux_torus}

First, we define the entrance time to a set $A\subset\Z^2_n$ by
\[
 T_n(A) = \min_{x\in A} T_n(x).
\]

Now, consider two sets $A\subset A'\subset \Z^2_n$, and
suppose that we are only interested in the trace left by the random
walk on the set~$A$. Then, 
(apart from the initial piece of the trajectory until 
hitting~$\partial A'$
for the first time)
it is enough to know 
the excursions of the random walk between the boundaries of~$A$ 
and~$A'$. By definition, an excursion~$\vr$ is a simple 
random walk path that starts at~$\partial A$ and ends
on its first visit to~$\partial A'$, i.e., $\vr=(\vr_0, \vr_1,\ldots,
\vr_m)$, where $\vr_0\in\partial A$, $\vr_m\in\partial A'$,
$\vr_k\notin\partial A'$ and $\vr_k\sim\vr_{k+1}$ for $k<m$. With some abuse 
of notation, we denote by $\vr_{\text{st}}:=\vr_0$ and
$\vr_{\text{end}}:=\vr_m$ the starting and the ending points
of the excursion.
To define these excursions, consider 
the following sequence of stopping times: 
\begin{align*}
 D_0 &= T_n(\partial A'),\\
 J_1 &= \inf\{t> D_0 : X_t \in \partial A\},\\
 D_1 &= \inf\{t> J_1 : X_t \in \partial A'\},
\end{align*}
and
\begin{align*}
 J_k &= \inf\{t> D_{k-1} : X_t \in \partial A\},\\
 D_k &= \inf\{t> J_k : X_t \in \partial A'\},
\end{align*}
for $k\geq 2$. Then,
denote by $Z^{(i)}=(X_{J_i}, \ldots, X_{D_i})$ the $i$th excursion
of~$X$ between~$\partial A$ and~$\partial A'$, for $i\geq 1$.
Also, let $Z^{(0)}=(X_0, \ldots, X_{D_0})$ be the ``initial''
excursion (it is possible, in fact, that it does not intersect
the set~$A$ at all). Recall that $t_\alpha:=\frac{4\alpha}{\pi}n^2\ln^2 n$
and define
\begin{align}
 N_\alpha &= \max\{k : J_k\leq t_\alpha\},\label{df_Na}
\\
 N'_\alpha &= \max\{k : D_k\leq t_\alpha\}\label{df_N'a}
\end{align}
to be the number of incomplete (respectively, complete)
excursions up to time~$t_\alpha$.

Next, we need also to define the excursions of random
interlacements 
in an analogous way. Assume that the trajectories 
of the $\s$-walks that intersect~$A$ are enumerated according
to their $u$-labels (recall the construction in Section~\ref{s_results}).
For each trajectory from that list (say, the $j$th one,
denoted~$\s^{(j)}$ and time-shifted in such a way 
that $\s^{(j)}_k\notin A$ for all $k\leq -1$ and   $\s^{(j)}_0 \in A$)
define the stopping times 
\begin{align*}
 {\hat J}_1 &= 0,\\
 {\hat D}_1 &= \inf\{t> {\hat J}_1 : \s^{(j)}_t \in \partial A'\},
\end{align*}
and
\begin{align*}
 {\hat J}_k &= \inf\{t> {\hat D}_{k-1} : \s^{(j)}_t \in \partial A\},\\
 {\hat D}_k &= \inf\{t> {\hat J}_k : \s^{(j)}_t \in \partial A'\},
\end{align*}
for $k\geq 2$. Let~$\ell_j=\inf\{k:{\hat J}_k=\infty\}-1$ be the
number of excursions corresponding to the $j$th trajectory.
The excursions of RI($\alpha$) between~$\partial A$ 
and~$\partial A'$ are then defined by
\[
 {\hat Z}^{(i)}=(\s^{(j)}_{{\hat J}_m}, \ldots, \s^{(j)}_{{\hat D}_m}),
\]
where 
$ i = m+\sum_{k=1}^{j-1} \ell_k$,
and $m=1,2,\ldots \ell_j$. We let $R_\alpha$ to be the number of
trajectories intersecting~$A$ and with labels less than~$\alpha \pi$,
and denote ${\hat N}_\alpha = \sum_{k=1}^{R_\alpha} \ell_k$
to be the total number of excursions of RI($\alpha$) between~$\partial A$
and~$\partial A'$.

Observe also that the above construction makes sense with $\alpha=\infty$
as well; we then obtain an infinite sequence of excursions
of RI (=RI($\infty$)) between~$\partial A$
and~$\partial A'$.

Finally, let us 
recall a result  of~\cite{DPRZ06} 
on the number of excursions 
for the simple random walk on~$\Z^2_n$.

\begin{lem}
\label{l_excursions_torus}
Consider the random variables $J_k,D_k$ defined in this section with 
$A=B\big(\frac{n}{3\ln n}\big)$, $A'=B\big(\frac{n}{3}\big)$.
Then, there exist positive constants $\delta_0, c_0 $ such that, for 
any $\delta$ with $ \frac{c_0}{\ln n} \leq \delta \leq \delta_0$,
we have
\begin{equation}
\label{eq_excursions_torus_g}
\IP\Big[J_k\in\Big((1-\delta)\frac{2 n^2\ln\ln n}{\pi}k,
(1+\delta)\frac{2 n^2\ln\ln n}{\pi}k\Big)\Big] 
    \geq 1 - \exp\big(-c\delta^2 k\big),
\end{equation}
and the same result holds with~$D_k$ on the place
of~$J_k$.
\end{lem}

\begin{proof}
 This is, in fact, a particular case of Lemma~3.2 of~\cite{DPRZ06}.
\end{proof}

We observe that~\eqref{eq_excursions_torus_g} means
that the ``typical'' number of excursions
by time~$s$ is $\frac{\pi s}{2n^2\ln\ln n}$.
In particular, a useful consequence
of Lemma~\ref{l_excursions_torus} is that for all uniformly 
positive~$\alpha$ and all large enough~$n$
\begin{equation}
\label{eq_excursions_torus}
\IP\Big[(1-\delta)\frac{2\alpha\ln^2 n}{\ln\ln n}\leq N_\alpha 
 \leq (1+\delta)\frac{2\alpha\ln^2 n}{\ln\ln n}\Big] 
    \geq 1 - \exp\Big(-c\delta^2\frac{\ln^2 n}{\ln\ln n}\Big),
\end{equation}
and the same result holds with $N'_\alpha$ on the place
of~$N_\alpha$,
where $N_\alpha,N'_\alpha$ are defined 
as in~\eqref{df_Na}--\eqref{df_N'a} with 
$A=B\big(\frac{n}{3\ln n}\big)$, $A'=B\big(\frac{n}{3}\big)$.

\section{Proofs of the main results}
\label{s_proofs}
We first prove the results related to random interlacements,
and then deal with the connections between random interlacements
and random walk on the torus in Section~\ref{s_proof_torus}.

\subsection{Proofs for random interlacements}
\label{s_proofs_interl}
First of all, we apply some results of Section~\ref{s_aux}
to finish the proof of Theorem~\ref{t_properties_RI}.
\begin{proof}[Proof of Theorem~\ref{t_properties_RI}, parts (iii)--(v).]
Recall the fundamental formula~\eqref{eq_vacant2} 
for the random interlacement and the relation~\eqref{formula_for_a}.
 Then, the  statement~(iv)
follows from Lemma~\ref{l_cap_triangle} 
and from (\ref{properties_RI_ii}), 
while~(v) is a consequence of Lemma~\ref{l_cap_distantball}~(i).

Finally, observe that, by symmetry, Theorem~\ref{t_properties_RI}~(ii),
 and Lemma~\ref{l_cap_distantball}~(ii) we have
\begin{align*}
\IP[A\subset\V^\alpha \mid x\in \V^{\alpha}] 
&= \exp\Big(-\pi\alpha\big(\capa(A\cup\{x\})-\capa(\{0,x\})\big)\Big)\\
&=\exp\Bigg(-\pi\alpha
\Big(\frac{\big(a(x)+O\big(\frac{r\ln r\ln\|x\|}{\|x\|}\big)\big)^2}
{2a(x)-\capa(A)+O\big(\frac{r\ln r\ln\|x\|}{\|x\|}\big)}
-\frac{a(x)}{2}\Big)\Bigg)\\
&=\exp\Bigg(-\frac{\pi\alpha}{4}\capa(A)
\frac{1+O\big(\frac{r\ln r \ln\|x\|}{\|x\|}\big)}
{1-\frac{\capa(A)}{2a(x)}
+O\big(\frac{r\ln r}{\|x\|}\big)}
\Bigg),
\end{align*}
thus proving the part~(iii).
\end{proof}

\begin{proof}[Proof of Theorem~\ref{t_sizevacant} (iii)]
 We start by observing that the first part of~(iii) follows
 from the bound~\eqref{eq_emptyball} and Borel-Cantelli. 
So, let us concentrate on proving~\eqref{eq_emptyball}.
Recall the following elementary
 fact: let~$N$ be a Poisson random variable with parameter~$\lambda$,
 and $Y_1,Y_2, Y_3,\ldots$ be independent (also of~$N$) random
 variables with exponential distribution ${\mathcal E}(p)$ with parameter~$p$.
Let~$\Theta=\sum_{j=1}^N Y_j$ be the corresponding
compound Poisson random variable. Its Cram\'er transform $b \mapsto \lambda(\sqrt{b}-1)^2$ is easily computed, and Chernov's bound gives, 
 for all $b>1$,
\begin{equation}
\label{CompPoisson_LD}
 \IP[\Theta \geq b\lambda p^{-1}] \leq \exp\big(-\lambda(\sqrt{b}-1)^2\big).
\end{equation}

Now, assume that $\alpha<1$.
Fix $\beta\in (0,1)$, 
which will be later taken close  to~$1$, and fix some set of non-intersecting 
disks~$B'_1=B(x_1,r^\beta),\ldots,B'_{k_r}=B(x_{k_r},r^\beta)\subset
B(r)\setminus B(r/2)$, with cardinality $k_r = \frac{1}{4}r^{2(1-\beta)}$.
Denote also $B_j:= B\big(x_j, \frac{r^\beta}{\ln^3 r^\beta}\big)$,
$j=1,\ldots,k_r$. 

Before going to the heart of the matter we briefly 
sketch the strategy of proof (also,
one may find it helpful to look at Figure~\ref{f_disks}). 
We start to show that at least a half of these balls $B_j$ will receive at most
$b\frac{2\alpha\ln^2 r}{3\ln\ln r^\beta}$ 
excursions from~$\partial B_{j}$ to~$\partial B'_{j}$ up to time $t_\alpha$, where 
$b>1$ is a parameter (the above number of excursions is larger than
the typical number of excursions by factor~$b$). 
Moreover, using the method of soft local times~\cite{CGPV,SLT}, we couple 
such excursions from $RI(\alpha)$ with 
a slightly larger number of independent excursions from the $\s$-walk: 
with overwhelming probability, the trace on 
$\cup_j B_j$ of the latter excursion process contains the trace of the former, 
so the vacant set~$\V^\alpha$ restricted to balls~$B_j$ is smaller than the set
of unvisited points by the independent process. Now, by independence, 
it will be possible to estimate the probability for leaving 
that many balls partially uncovered, and this will conclude the proof. 
\begin{figure}
\begin{center}
\includegraphics{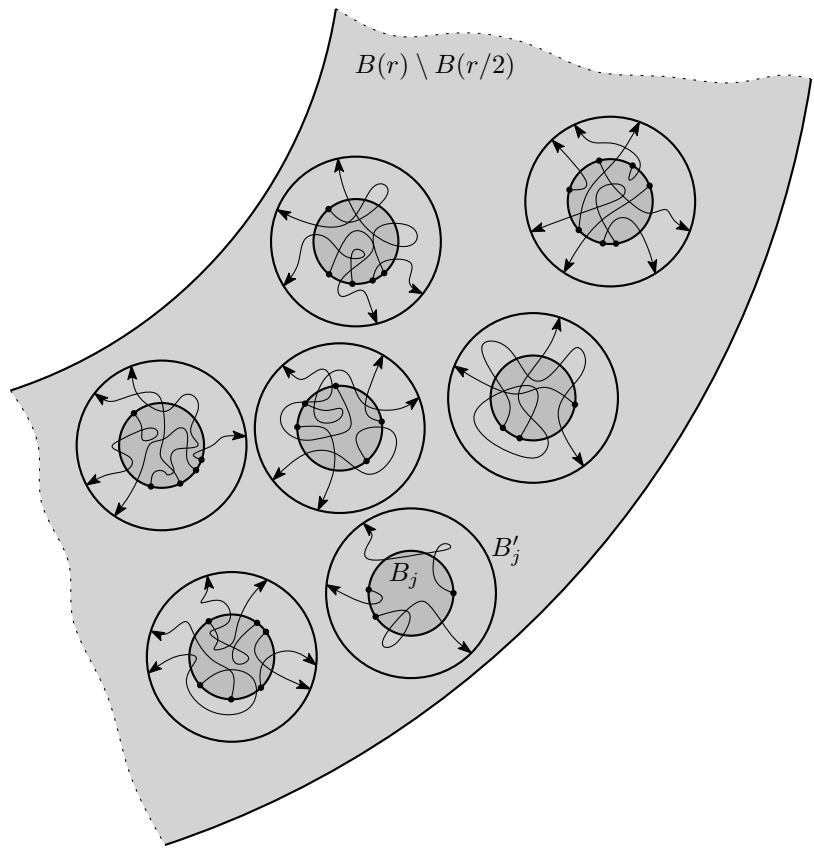}
\caption{On the proof of Theorem~\ref{t_sizevacant} (iii). With 
high probability, at least a positive proportion of the inner circles 
is not completely covered.}
\label{f_disks}
\end{center}
\end{figure}

Let us observe that the number of $\s$-walks in $RI(\alpha)$ 
intersecting a given disk~$B_{j}$
has Poisson law with parameter 
$\lambda=(1+o(1))\frac{2\alpha}{2-\beta}\ln r$. 
Indeed, the law is Poisson by construction, 
the parameter $\pi \alpha \capa (B_{j} \cup \{0\})$ is found 
in~\eqref{eq_vacant2} and 
then estimated  using
Lemma~\ref{l_cap_distantball} (i).

Next, by Lemma~\ref{l_escape_from_ball} (i), the probability that
the walk $\s$ started from any $y\in\partial B'_{j}$
does not hit~$B_{j}$ is $(1+o(1))\frac{3\ln\ln r^\beta}{(2-\beta)\ln r}$. 
This depends on the starting point, 
however, after the first visit to~$\partial B_{j}$,
 each $\s$-walk generates a number of excursions 
between~$\partial B_{j}$ and~$\partial B'_{j}$ which is dominated by a 
geometric law $G(p')$ (supported on $\{1,2,3,\ldots\}$)
 with success parameter 
$p'=(1+o(1))\frac{3\ln\ln r^\beta}{(2-\beta)\ln r}$.
Recall also that the integer part of ${\mathcal E}(u)$ 
is geometric $G(1-e^{-u})$.
So, with $p=-\ln (1-p')$, the total number ${\hat N}_\alpha^{(j)}$ of excursions 
between~$\partial B_{j}$ and~$\partial B'_{j}$ in RI($\alpha$) 
can be dominated by a compound
Poisson law with ${\mathcal E}(p)$ terms in the sum 
with expectation 
\[
\lambda p^{-1}
  = (1+o(1))\frac{2\alpha\ln^2 r}{3\ln\ln r^\beta}.
\]
Then, using~\eqref{CompPoisson_LD}, we obtain for $b>1$
\begin{align}
\IP\Big[{\hat N}_\alpha^{(j)}\geq b\frac{2\alpha\ln^2 r}
{3\ln\ln r^\beta}\Big] &\leq
\exp\Big(-(1+o(1))\big(\sqrt{b}-1\big)^2 \frac{2\alpha}{2-\beta}
\ln r\Big) \nonumber\\
&=  r^{-(1+o(1))(\sqrt{b}-1)^2 \frac{2\alpha}{2-\beta}}.
\label{LD_numb_exc}
\end{align}
Now, let~$W_b$ be the set 
\[
 W_b = \Big\{j\leq k_r: {\hat N}_\alpha^{(j)} 
< b\frac{2\alpha\ln^2 r}{3\ln\ln r^\beta}\Big\}.
\]
Combining~\eqref{LD_numb_exc} with Markov inequality, we obtain
\begin{equation*}
\frac{k_r}{2} \IP\big[|\{1,\ldots,k_r\}\setminus W_b| > k_r/2\big] 
\leq \E \big\vert \{1,\ldots,k_r\}\setminus W_b \big\vert 
\leq 
k_r r^{-(1+o(1))(\sqrt{b}-1)^2
\frac{2\alpha}{2-\beta}},
\end{equation*}
so
\begin{equation}
\label{manydisks}
\IP\big[|W_b|\geq k_r/2\big] \geq 1-2r^{-(1+o(1))(\sqrt{b}-1)^2
\frac{2\alpha}{2-\beta}}.
\end{equation}
Assume that~$1<b<\alpha^{-1}$ and~$\beta\in (0,1)$ is close enough to~$1$,
so that $\frac{b\alpha}{\beta^2}<1$. 
As in Section~\ref{s_aux_torus}, we denote by
${\hat Z}^{(1),j},\ldots,{\hat Z}^{({\hat N}_\alpha^{(j)}),j}$
the excursions of RI($\alpha$)
between~$\partial B_j$ and~$\partial B'_j$. Also, let
$ {\tilde Z}^{(1),j},{\tilde Z}^{(2),j},{\tilde Z}^{(3),j},\ldots$
be a sequence of i.i.d.\ $\s$-excursions 
between~$\partial B_j$ and~$\partial B'_j$, started with the law
 $\hm_{B_j}$; these sequences themselves are also assumed to 
 be independent. 
 
Abbreviate $m= b\frac{2\alpha\ln^2 r}{3\ln\ln r^\beta}$. Next, 
for $j=1,\ldots,k_r$ we consider the events
\[
 D_j = \Big\{\big\{{\hat Z}^{(1),j},\ldots,
{\hat Z}^{({\hat N}_\alpha^{(j)}),j}\big\}
\subset \big\{{\tilde Z}^{(1),j}, \ldots, 
{\tilde Z}^{((1+\delta)m),j}\big\}\Big\}.
\]

\begin{lem}
\label{l_SLT_coupling}
 It is possible to construct the excursions 
 $\big({\hat Z}^{(k),j}, k=1,\ldots, {\hat N}_\alpha^{(j)}\big)$,
 $\big({\tilde Z}^{(k),j}, k=1,2,3,\ldots\big)$, $j=1,\ldots,k_r$,
on a same probability space in such a way that
for a fixed $C'>0$
\begin{equation}
\label{D_j_complementary}
 \IP\big[D_j^\complement\big] 
\leq \exp\Big(- C' \frac{\ln^2 r}{(\ln\ln r)^2}\Big),
\end{equation}
for all $j=1,\ldots,k_r$.
\end{lem}

\begin{proof}
This coupling can be built using the method
of \emph{soft local times} of~\cite{SLT}; in this specific
situation, the exposition of~\cite{CGPV} is better suited.
Roughly speaking, this method consists of using a marked
Poisson point process on $\big(\bigcup_j\partial B_j\big)\times \R_+$,
where the ``marks'' are corresponding excursions; see Section~2 of~\cite{CGPV}
for details. Then, the ${\tilde Z}$-excursions are simply the marks of the 
points of the Poisson processes on $\partial B_j\times \R_+$ ordered
according to the second coordinate, so they are independent by
construction. The ${\hat Z}$-excursions are also the marks of the 
points of these Poisson processes, but generally taken in a different
order using a special procedure; Figure~1 of~\cite{CGPV} is, hopefully,
self-explanating and may provide some quick insight.

The inequality~\eqref{D_j_complementary} then follows from
Lemma~2.1 of~\cite{CGPV} (observe that, by 
Lemma~\ref{l_entrance_hat_s}, the parameter~$v$ in Lemma~2.1 of~\cite{CGPV}
can be anything exceeding $O(1/\ln^2 r)$, so we
choose e.g.\ $v=(\ln\ln r)^{-1}$; also, that lemma is clearly
valid for $\s$-excursions as well).
\end{proof}

We continue the proof of part~(iii) of Theorem~\ref{t_sizevacant}. Define
\[
 D = \bigcap_{j\leq k_r} D_j;
\]
using~\eqref{D_j_complementary}, we obtain by the union bound the subpolynomial estimate
\begin{equation}
\label{est_D}
\IP\big[D^\complement\big] 
\leq \frac{1}{4}r^{2(1-\beta)}
\exp\Big(- C' \frac{\ln^2 r}{(\ln\ln r)^2}\Big).
\end{equation}

Let~$\delta>0$ be such that $(1+\delta)\frac{b\alpha}{\beta^2}<1$.
Define the events
\[
 {\tilde\G}_j = \big\{B'_j\text{ is completely covered by }
  {\tilde Z}^{(1),j}\cup\cdots\cup{\tilde Z}^{((1+\delta)m),j}\big\}.
\]
Then, for all~$j\leq k_r$ it holds that
\begin{equation}
\label{1/5}
 \IP[{\tilde\G}_j]\leq \frac{1}{5}
\end{equation}
for all large enough~$r$. Indeed, if the $\tilde Z$'s were 
independent SRW-excursions,
the above inequality (with any fixed constant in the right-hand
side) could be obtained as in
the proof of Lemma~3.2 of~\cite{CGPV} 
(take ${\tilde H}=\hm_{B_j}$ and $n=3r^\beta + 1$ there). On the other hand, 
Lemma~\ref{l_relation_S_hatS}~(ii) implies that the first~$(1+\delta)m$
$\s$-excursions can be coupled with SRW-excursions with
high probability, so~\eqref{1/5} holds
for $\s$-excursions as well.

Next, define the set
\[
 {\widetilde W} = \big\{j\leq k_r : {\tilde\G}_j^\complement
  \text{ occurs}\big\}.             
\]
Since the events $({\tilde\G}_j, j\leq k_r)$ are independent, 
by~\eqref{1/5} we have (recall that $k_r = \frac{1}{4}r^{2(1-\beta)}$)
\begin{equation}
\label{est_widetilde_W}
 \IP\Big[|{\widetilde W}| \geq \frac{3}{5}k_r\Big]
 \geq 1 - \exp\big(-C r^{2(1-\beta)}\big) 
\end{equation}
for all~$r$ large enough. 

Observe that, by construction, on the event~$D$
we have $\V^\alpha\cap B'_j\neq \emptyset$
for all $j\in {\widetilde W}\cap W_b$. 
So, using~\eqref{manydisks}, \eqref{est_widetilde_W},
and~\eqref{est_D}, we obtain
\begin{align*}
\lefteqn{ \IP\Big[\V^\alpha\cap \big(B(r)\setminus B(r/2)\big)
=\emptyset\Big]}\\ 
&\leq 
 \IP\Big[|W_b|< \frac{k_r}{2}\Big] + 
\IP\big[D^\complement\big] 
+ \IP\Big[|{\widetilde W}|< \frac{3k_r}{5}\Big]
\\ 
  &\leq  
  2r^{-(1+o(1))(\sqrt{b}-1)^2
\frac{2\alpha}{2-\beta}}
+\exp\big(-C r^{2(1-\beta)}\big) 
+ \frac{1}{4}r^{2(1-\beta)}
\exp\Big(- C' \frac{\ln^2 r}{(\ln\ln r)^2}\Big).
\end{align*}
Since $b\in (1,\alpha^{-1})$ can be arbitrarily close
to~$\alpha^{-1}$ and $\beta\in (0,1)$ 
can be arbitrarily close to~$1$,
this concludes the proof of (\ref{eq_emptyball}).
\end{proof}

\begin{proof}[Proof of Theorem~\ref{t_sizevacant} (ii)]
To complete the proofs in Section~\ref{s_results}, it remains to prove that $|\V^\alpha|<\infty$
a.s.\ for $\alpha>1$.
 First,we establish the following elementary fact. 
For $x\in\partial B(2r)$ and $y\in B(r)\setminus B(r/2)$, it holds
\begin{equation}
\label{exc_r_2r}
 \hP_x\big[\htau_1(B(y))<\htau_1\big(B(r\ln r)\big)\big] 
= \frac{\ln\ln r}{\ln r}\big(1+o(1)\big).
\end{equation}
 Indeed, define the events
\begin{align*}
 G_0 &= \{\tau_1(B(0))<\tau_1\big(B(r\ln r)\big)\},\\
 G_1 &= \{\tau_1(B(y))<\tau_1\big(B(r\ln r)\big)\};
\end{align*}
then, Lemma~\ref{l_exit_balls} implies that 
\[
 P_x[G_0] = \frac{\ln\ln r}{\ln r}\big(1+o(1)\big) = P_x[G_1] \big(1+o(1)\big) .
\]
Observe that 
\begin{align*}
  P_x[G_0 \cap G_1] &= P_x[G_0 \cup G_1] P_x[G_0 \cap G_1\mid G_0 \cup G_1]\\
&\leq \big(P_x[G_0]+P_x[G_1]\big)\big(P_0[G_1]+P_y[G_0]\big)\\
&\leq 4\Big(\frac{\ln\ln r}{\ln r}\Big)^2 \big(1+o(1)\big).
\end{align*}
So,
\begin{align*}
 P_x[G_1\mid G_0^\complement] 
&= \frac{P_x[G_1]-P_x[G_0\cap G_1]}{1-P_x[G_0]}\\
  &= \frac{\ln\ln r}{\ln r}\big(1+o(1)\big),
\end{align*}
and we use Lemma~\ref{l_relation_S_hatS} to conclude
the proof of~\eqref{exc_r_2r}.

Now, the goal is to prove that, for $\alpha>1$
\begin{equation}
\label{covered_ring}
 \IP\big[\text{there exists }y\in B(r)\setminus B(r/2)
\text{ such that }y\in\V^\alpha\big] 
\leq r^{-\frac{\alpha}{2}(1-\alpha^{-1})^2(1+o(1))}.
\end{equation}
This would clearly imply that the set~$\V^\alpha$ is a.s.\
finite, since
\begin{equation}
\label{BorCan}
 \{\V^\alpha\text{ is infinite}\}
= \big\{\V^\alpha\cap\big(B(2^n)\setminus B(2^{n-1})\big)\neq \emptyset
\text{ for infinitely many }n\big\},
\end{equation}
and the Borel-Cantelli lemma together with~\eqref{covered_ring}
imply that the probability of the latter event equals~$0$.

Let~$N_{\alpha,r}$ be the number of $\s$-excursions of RI($\alpha$)
between~$\partial B(r)$ and $\partial B(r\ln r)$. 
Analogously to~\eqref{LD_numb_exc} (using Lemma~\ref{l_escape_hatS} 
in place of Lemma~\ref{l_escape_from_ball}~(i)), 
it is straightforward
to show that, for $b<1$, 
\begin{equation}
\label{LD_exc_b<1}
\IP\Big[N_{\alpha,r} \leq b\frac{2\alpha\ln^2 r}{\ln\ln r}\Big] 
\leq r^{-2\alpha(1-\sqrt{b})^2 (1+o(1))}.
\end{equation}
Now, \eqref{exc_r_2r} implies that for $y\in B(r)\setminus B(r/2)$
\begin{align}
 \IP\Big[y\text{ is uncovered by first }
 b\frac{2\alpha\ln^2 r}{\ln\ln r}\text{ excursions}\Big] 
& \leq \Big(1-\frac{\ln\ln r}{\ln r}
(1+o(1))\Big)^{b\frac{2\alpha\ln^2 r}{\ln\ln r}}\nonumber\\
& = r^{-2b\alpha(1+o(1))},
\label{hit_y_in_the_ring}
\end{align}
so, using the union bound,
\begin{equation}
\label{cover_ring_b}
 \IP\Big[\exists y\in B(r)\setminus B(r/2):
y\in\V^\alpha , N_{\alpha,r} >
 b\frac{2\alpha\ln^2 r}{\ln\ln r}
\Big] 
\leq r^{-2(b\alpha-1)(1+o(1))}.
\end{equation}
Using~\eqref{LD_exc_b<1} and~\eqref{cover_ring_b} with 
$b=\frac{1}{4}\big(1+\frac{1}{\alpha}\big)^2$ we
conclude the proof of~\eqref{covered_ring} and
of Theorem~\ref{t_sizevacant}~(ii).
\end{proof}

\subsection{Proof of Theorem~\ref{t_conditional}}
\label{s_proof_torus}
Let us first give a more detailed
heuristic argument for~\eqref{eq_conditional}.
As usual, we consider the \emph{excursions} of the random walk~$X$
between $\partial B\big(\frac{n}{3\ln n}\big)$ and $\partial B(n/3)$
up to time~$t_\alpha$.
Recall that~$N_\alpha$ denotes the number of these excursions. 
Lemma~\ref{l_excursions_torus} shows
that this (random) number is concentrated around 
$\frac{2\alpha\ln^2 n}{\ln\ln n}$, with deviation probabilities
of subpolynomial order.
This is for unconditional probabilities, but, since the probability
of the event $\{0\in U_{t_\alpha}^{(n)}\}$ is only polynomially
small (actually, it is $n^{-2\alpha+o(1)}$), the same holds 
for the deviation probabilities conditioned on this event. 
So, let us just assume for now that the number of the excursions
is \emph{exactly} $\frac{2\alpha\ln^2 n}{\ln\ln n}$ a.s.,
 and see where will it lead us.

Assume without restriction of generality that $0\in A$.
Then, Lemmas~\ref{l_exit_balls} and~\ref{l_hit_A} imply that
\begin{itemize}
 \item the probability that an excursion hits the origin
is roughly $\frac{\ln\ln n}{\ln (n/3)}$;
 \item provided that $\capa(A)\ll \ln n$, 
the probability that an excursion hits the set~$A$
is roughly $\frac{\ln\ln n}{\ln (n/3)}
\big(1+\frac{\pi\capa(A)}{2\ln (n/3)}\big)$.
\end{itemize}
So, the \emph{conditional} probability~$p_*$ 
that an excursion does not hit~$A$ given that it does not 
hit the origin is
\[
 p_*\approx \frac{1-\frac{\ln\ln n}{\ln (n/3)}
\big(1+\frac{\pi\capa(A)}{2\ln (n/3)}\big)}{1-\frac{\ln\ln n}{\ln (n/3)}}
\approx 1 - \frac{\pi\ln\ln n}{2\ln^2 n}\capa(A),
\]
and then we obtain
\begin{align*}
 \IP[\Upsilon_n A \subset 
   U_{t_\alpha}^{(n)} \mid 0\in U_{t_\alpha}^{(n)}] 
&\approx p_*^{N_\alpha}\\
&\approx \Big(1 - 
\frac{\pi\ln\ln n}{2\ln^2 n}\capa(A)\Big)^{\frac{2\alpha\ln^2 n}{\ln\ln n}}
\\
&\approx \exp\big(-\pi\alpha\capa(A)\big),
\end{align*}
which agrees with the statement of Theorem~\ref{t_conditional}. 

However, turning the above heuristics 
to a rigorous proof is not an easy task. The reason for this is that,
although~$N_\alpha$ is indeed concentrated around
$\frac{2\alpha\ln^2 n}{\ln\ln n}$, it is not \emph{concentrated enough}:
the probability that~$0$ is not 
hit during~$k$ excursions, where~$k$ varies over the ``typical''
values of~$N_\alpha$, changes too much. Therefore,
in the proof of Theorem~\ref{t_conditional} we take a different
route by considering the suitable $h$-transform of the walk, 
as explained below.

Define for $x\in\Z^2_n$ and \emph{any}~$t$ 
\[
h(t,x)=P_x[T_n(0)>t]
\]
(so that $h(t,x)=1$ for $t<0$).
To simplify the notations, let us also assume that~$t_\alpha$
is integer.
We will represent the conditioned random walk
as a time-dependent Markov chain, using
the Doob's $h$-transform. 
Indeed, it is well known and easily checked 
that the simple random
walk on~$\Z_n^2$ conditioned on the event $\{0\in U^{(n)}_{t_\alpha}\}$ 
is a time-dependent
Markov chain $\tX$ with transition 
probabilities given by
\begin{equation}
\label{df_h_transf}
 \IP[\tX_{s+1}=y \mid \tX_s=x]
   = \frac{h(t_\alpha-s-1,y)}{h(t_\alpha-s,x)}\times \frac{1}{4},
\end{equation}
if $x$ and~$y$ are neighbours, 
and equal to 0 otherwise.  For simpler notations, we do not indicate 
the dependence on~$t_\alpha$ in the notation $\tilde X$.
In order to  proceed, we need the following fact 
(its proof can be skipped on a first reading).
\begin{lem}
\label{l_reg_h}
For all $\lambda \in (0,1/5)$, there exist $c_1>0, n_1 \geq 2$,
$\sigma_1>0$ (depending on~$\lambda$) such that
for all~$n \geq n_1$, $1\leq \beta\leq \sigma_1\ln n$,  
 $\|x\|,\|y\|\geq \lambda n$, 
$|r|\leq \beta n^2$ and  all~$s\geq 0$,
\begin{equation}
\label{regularity_h}
\Big|\frac{h(s,x)}{h(s+r,y)}-1\Big| \leq \frac{c_1\beta}{\ln n} \;.
\end{equation}
\end{lem}

\begin{proof}
Denote 
\[
 h(t, \mu)  :=  P_{\mu}[T_n(0)>t],
\]
where $P_{\mu}[\cdot]$ is the probability for 
the simple random walk on $\Z_n^2$ starting from the initial distribution~$\mu$. 

Using the local CLT for the two-dimensional SRW (e.g., Theorem~2.3.11 of~\cite{LL10})
it is straightforward to obtain that for a large enough~$\kappa>2$ and all $t>n^2$ and
$x\in\Z^2_n$
\begin{equation}
\label{LCLT}
 P_0[X_t = x] \leq \frac{\kappa}{n^2}.
\end{equation}

 Let us define the set~$\M$ of probability
measures on~$\Z^2_n$ in the following way:
\[
 \M = \big\{\nu : \nu\big(B(j)\big)\leq 7\kappa j^2n^{-2} 
\text{ for all }j\leq \lambda n\big\};
\]
observe that any probability
measure concentrated on a one-point set~$\{x\}$ with $\|x\|\geq \lambda n$
belongs to~$\M$.
Assume from now on that~$n$ is odd, so that the simple random walk on~$\Z^2_n$
is aperiodic (the case of even~$n$ can be treated essentially 
in the same way, with some obvious modifications). Recall that the uniform measure~$\mu_0$ on~$\Z^2_n$, i.e., $\mu_0(x)=n^{-2}$ for all~$x\in\Z^2_n$,
is the invariant law of simple random walk  $X$ on~$\Z^2_n$. 
It is straightforward to observe that $\mu_0\in\M$;
moreover, it holds in fact that
 $\mu_0\big(B(j)\big)\leq 7 j^2n^{-2} 
\text{ for all }j\leq \lambda n$ (i.e., with~$\kappa=1$).
Thus, for any probability measure~$\nu$ such that 
$\nu(x)\leq \kappa\mu_0(x)$ for all~$x\in B(\lambda n)$
it holds that $\nu\in\M$. So, \eqref{LCLT} implies that
\begin{equation}
\label{belongs_to_M}
 \nu P^{(t)}\in \M \quad \text{ for any }\nu \text{ and all }t\geq n^2. 
\end{equation}

Let us abbreviate $\nu P^{(s)}(\cdot)=P_\nu [X_s\in \cdot\,]$.
Recall that the mixing time of~$X$ is of order~$n^2$ (e.g., Theorem~5.5 
in~\cite{LPW}).
Using the bound on the separation distance provided by
 Lemma~19.3 of~\cite{LPW},
it is clear that for any $\eps\in (0,1)$ one can find large enough~$c'$
(in fact, $c'=O(\ln \eps^{-1})$)
such that for any probability measure~$\nu$
it holds that $\nu P^{(s)}\geq (1-\eps)\mu_0$ for  all $s\geq (c'-1) n^2$.
Using~\eqref{belongs_to_M}, we obtain for all $s\geq c' n^2$,
\begin{eqnarray}
 \nu P^{(s)} 
&=&    (1-\eps)\mu_0 P^{(n^2)} +\big(  \nu P^{(s\!-\!n^2)}-(1-\eps)\mu_0\big)  P^{(n^2)} 
\nonumber \\ \label{spectralgap}
&=&
 (1-\eps)\mu_0 + \eps\nu', \quad \text{ with } \nu'\in \M.
\end{eqnarray}

We are now going to obtain that there exists some~$c_2>0$
such that for all~$b \in \{1,2,3,\ldots\}$ and all $\nu \in \M$,
\begin{equation}
\label{entre_0_nicht}
 h(bn^2,\nu) = P_\nu [T_n(0) > bn^2] \geq 1 - \frac{bc_2}{\ln n}.
\end{equation}
To prove~\eqref{entre_0_nicht}, let us first show that there exists~$c_3=c_3(\lambda)>0$ such that
\begin{equation} 
\label{escape_to_lambda_n}
 P_\nu\big[T_n(0) < T_n\big(\partial B(\lambda n)\big)\big]
 \leq \frac{c_3}{\ln n}
\end{equation}
for all $\nu \in \M$. Abbreviate 
$W_j=B\big(\frac{\lambda n}{2^{j-1}}\big)\setminus 
B\big(\frac{\lambda n}{2^j}\big)$ and write, 
using~\eqref{nothit_0}
\begin{align*}
 \sum_{x\in W_j} \nu(x)
P_x\big[T_n(0) < T_n\big(\partial B(\lambda n)\big)\big]
& \leq \sum_{x\in W_j}\nu(x)\Big(1-\frac{a(x)}{a(\lambda n)+O(n^{-1})}\Big)\\
& \leq 7\kappa n^{-2}\times \frac{\lambda^2n^2}{2^{2(j-1)}}\times
\Big(1-\frac{\ln \lambda n - j\ln 2}{\ln \lambda n+O(n^{-1})}\Big)\\
&\leq \frac{1}{\ln\lambda n} \times \frac{7\kappa j\lambda^2\ln 2}{2^{2(j-1)}}.
\end{align*}
Let $j_0$ be such that $\frac{\lambda n}{2^{j_0}}\leq \frac{n}{\sqrt{\ln n}}$.
Write  
\begin{align*}
\lefteqn{P_\nu\big[T_n(0) < T_n\big(\partial B(\lambda n)\big)\big]}\\
&\leq \nu\Big(B\Big(\frac{n}{\sqrt{\ln n}}\Big)\Big)
+ \sum_{j=1}^{j_0} \sum_{x\in W_j} \nu(x)
P_x\big[T_n(0) < T_n\big(\partial B(\lambda n)\big)\big]\\
&\leq \frac{7\kappa }{\ln n} + \frac{1}{\ln\lambda n} 
\sum_{j=1}^{j_0} \frac{7\kappa j\lambda^2\ln 2}{2^{2(j-1)}}
\end{align*}
for $\nu \in \M$, 
which proves~\eqref{escape_to_lambda_n}. To obtain~\eqref{entre_0_nicht},
let us recall that $\nu P^{(n^2)}\in\M$
for any~$\nu$ by~\eqref{belongs_to_M}. 
Observe that the number of excursions
by time~$n^2$
between $\partial B(\lambda n)$ and $\partial B(n/3)$ 
is stochastically bounded by a Geometric random variable
with expectation of constant order. Since (again by~\eqref{nothit_0})
for any $x\in \partial B(\lambda n)$
\[
 P_x\big[T_n(0) < T_n\big(\partial B(n/3)\big)\big] 
   \leq \frac{c_4}{\ln n},
\]
for some $c_4=c_4(\lambda)$,
and, considering a random sum of a geometric number of 
independent Bernoulli with parameter $c_4/\ln n$, 
using also~\eqref{escape_to_lambda_n} it is not difficult to obtain that
for any $\nu\in \M$
\begin{equation}
\label{for_b_0}
 P_\nu [T_n(0) \leq n^2] \leq \frac{c_5}{\ln n}.
\end{equation}
The inequality~\eqref{entre_0_nicht} then follows from~\eqref{for_b_0}
and the union bound,
\begin{equation}
\label{n_entre_pas_0}
P_\nu [T_n(0) \leq k n^2] \leq \sum_{j=0}^{k-1} 
P_{\nu P^{(jn^2})} [T_n(0) \leq n^2] \leq \frac{kc_5}{\ln n}.
\end{equation}

Now, let $c' \in \{1,2,3,\ldots\} $
 be such that $\eps < 1/3$ in~\eqref{spectralgap}.
Assume also that~$n$ is sufficiently large so that
 $\big(1-\frac{c'c_2}{\ln n}\big)^{-1}\leq 2$.
Then, \eqref{entre_0_nicht} implies that for all $s\leq c'n^2$ and 
$\nu\in\M$
\[
 \frac{h(s,\mu)}{h(s,\nu)}-1 \leq \frac{1}{1-\frac{c'c_2}{\ln n}}-1
=\Big(1-\frac{c'c_2}{\ln n}\Big)^{-1} \times\frac{c'c_2}{\ln n},
\]
and therefore
\begin{align}
 \frac{h(s,\mu)}{h(s,\nu)}-1 &\leq \frac{3c'c_2}{\ln n}
\qquad \text{ for any } \nu\in\M \text{ and arbitrary }\mu,
\label{whatweneed}\\
 1-\frac{h(s,\mu)}{h(s,\nu)} &\leq \frac{3c'c_2}{\ln n}
\qquad \text{ for any } \mu\in\M \text{ and arbitrary }\nu.
\label{whatweneed2}
\end{align}

We now extend~\eqref{whatweneed}--\eqref{whatweneed2} 
 from times $s \leq s_0=c'n^2$ to all times using
induction. Let $s_k=(k+1)c'n^2$, 
and consider the recursion hypothesis
\[
(H_k) : \qquad  \eqref{whatweneed}\ {\rm and } \ \eqref{whatweneed2} \; {\rm hold\ for\ } \ s \leq s_k,
\]
that we just have proved for $k=0$. Assume now $(H_k)$ for some $k$. 
Define the event $G_{r,s}=\{X_j\neq 0 \text{ for all }r+1\leq j\leq s\}$, 
and write
\begin{align}
 h(s+t,\mu) &= P_\mu[G_{t,s+t}]
 P_\mu[T_n(0)>t\mid G_{t,s+t}]\nonumber\\
&= h(s,\mu P^{(t)}) P_\mu[T_n(0)>t\mid G_{t,s+t}].
\label{Markov_future}
\end{align}
Abbreviate  $t= c'n^2$ for the rest of the proof of the Lemma.
Let us estimate the second term in the right-hand side
 of~\eqref{Markov_future}. Let~$\Gamma_{[0,t]}$ be the set 
of all nearest-neighbour trajectories on~$\Z^2_n$ of length~$t$.
For $\vr\in \Gamma_{[0,t]}$ we have 
$P_\mu[\vr]=\mu(\vr_0)\big(\frac{1}{4}\big)^{|\vr|}$ and
\[
 P_\mu[\vr\mid G_{t,s+t}] = \mu(\vr_0)\Big(\frac{1}{4}\Big)^{|\vr|}
\times \frac{h(s,\vr_{\text{end}})}{h(s, \mu P^{(t)})}
\leq \mu(\vr_0)\Big(\frac{1}{4}\Big)^{|\vr|} 
\Big(1+\frac{3c'c_2}{\ln n}\Big)
\]
using the relation~\eqref{whatweneed} for $s \leq s_k$.
Summing over  $\vr$ such that $T_n(0)\leq t$ and using~\eqref{entre_0_nicht}, we obtain, for $\mu \in \M$, 
\begin{equation}
\label{time_rev_implies}
 P_\mu[T_n(0)>t\mid G_{t,s+t}] 
  \geq 1 - \frac{c'c_2}{\ln n}\Big(1+\frac{3c'c_2}{\ln n}\Big) .
\end{equation}

Now, we use~\eqref{spectralgap} and~\eqref{Markov_future} 
to obtain that, with $\mu', \nu'$ defined in~\eqref{spectralgap},  
\begin{align*}
\frac{h(s+t,\mu)}{h(s+t,\nu)} &= 
\frac{h(s,\mu P^{(t)})P_\mu[T_n(0)>t\mid G_{t,s+t}]}
{h(s,\nu P^{(t)})P_\nu[T_n(0)>t\mid G_{t,s+t}]}\\
&=\frac{\big(1-\eps+\eps\frac{h(s,\mu')}{h(s,\mu_0)}\big)
P_\mu[T_n(0)>t\mid G_{t,s+t}]}
{\big(1-\eps+\eps\frac{h(s,\nu')}{h(s,\mu_0)}\big)
P_\nu[T_n(0)>t\mid G_{t,s+t}]},
\end{align*}
for $s \leq s_k$. 
We now use $(H_k)$ for the two ratios of $h$'s in the above expression, we also use~\eqref{time_rev_implies}
for the conditional probability in the denominator --  simply bounding it by~$1$ in the numerator -- 
to obtain 
\begin{align*}
\frac{h(s+t,\mu)}{h(s+t,\nu)}
\leq \frac{\big(1-\eps+\eps\big(1+\frac{3c'c_2}{\ln n}\big)\big)}
{\big(1-\eps+\eps\big(1-\frac{3c'c_2}{\ln n}\big)\big)
\big(1 - \frac{c'c_2}{\ln n}\big(1+\frac{3c'c_2}{\ln n}\big)\big)},
\end{align*}
that is,
\[
\frac{h(s+t,\mu)}{h(s+t,\nu)}-1
\leq (6\eps+1)\frac{c'c_2}{\ln n} + o\big((\ln n)^{-1}\big)
\]
for $\nu\in \M$.
Since $\eps<1/3$, for large enough~$n$ we obtain
that~\eqref{whatweneed}
also holds for all $s\leq s_{k+1}$.
In the same way, we prove the validity of~\eqref{whatweneed2} 
 for $s\leq s_{k+1}$.

This proves the recursion, which 
in turn implies~\eqref{regularity_h} for the case~$r=0$.
To treat the general case, observe that 
\begin{equation}
\label{h_decomposition}
h(s+r,y)=h(r,y)h(s,\nu), \quad \text{where }
\nu(\cdot)=P_y[X_r=\cdot\mid T_n(0)>r].
\end{equation}
Note that, by~\eqref{n_entre_pas_0}, we can choose~$\sigma_1$
in such a way that $P_y[T_n(0)>r]\geq \frac{1}{2}$.
Now, without loss of generality, we can assume that $r\geq {\tilde c}n^2$,
where~${\tilde c}$ is such that $2P_y[X_t=\cdot\,]\in\M$ for all 
$t\geq {\tilde c}n^2$
(clearly, such~${\tilde c}$ exists; e.g., consider~\eqref{spectralgap}
with $\eps=1/2$). Then, the general case in~\eqref{regularity_h}
follows from~\eqref{n_entre_pas_0} and~\eqref{h_decomposition}.
\end{proof}

Now we are able to prove Theorem~\ref{t_conditional}.

\begin{proof}[Proof of Theorem~\ref{t_conditional}]
Abbreviate $\delta_{n,\alpha} = C\alpha\sqrt{\frac{\ln\ln n}{\ln n}}$ 
and 
\[
 I_{\delta_{n,\alpha}} = \Big[(1-\delta_{n,\alpha})
\frac{2\alpha\ln^2 n}{\ln\ln n},
    (1+\delta_{n,\alpha})\frac{2\alpha\ln^2 n}{\ln\ln n}\Big].
\]
Let $N_\alpha$ be the number of excursions 
between $\partial B\big(\frac{n}{3\ln n}\big)$ and $\partial B(n/3)$
up to time~$t_\alpha$.
It holds that
 $\IP\big[0\in U^{(n)}_{t_\alpha}\big] =n^{-2\alpha+o(1)}$, 
see e.g.\ (1.6)--(1.7) in \cite{CGPV}. Then,
observe that~\eqref{eq_excursions_torus} implies that
\begin{align*}
 \IP\big[N_\alpha\notin I_{\delta_{n,\alpha}} \; \big|\;
 0\in U^{(n)}_{t_\alpha}\big] 
\leq \frac{\IP[N_\alpha\notin I_{\delta_{n,\alpha}}]}
{\IP[0\in U^{(n)}_{t_\alpha}]} \leq n^{2\alpha+o(1)}
 \times n^{-C'\alpha^2},
\end{align*}
where~$C'$ is a constant that can be made arbitrarily
large by making the constant~$C$ in the definition of~$\delta_{n,\alpha}$
 large enough.
So, if~$C$ is large enough, for some $c''>0$ it holds that
\begin{equation}
\label{cond_numb_exc}
\IP\big[N_\alpha\in I_{\delta_{n,\alpha}} \; 
\big|\; 0\in U^{(n)}_{t_\alpha}\big] 
  \geq 1 - n^{-c''\alpha}.
\end{equation}

We assume that the set~$A$ is fixed, so that $\capa(A)=O(1)$ and diam$(A)=O(1)$.
In addition, assume without loss of generality that $0\in A$.
Recall that
with~\eqref{cond_numb_exc} we control the number of excursions
between $\partial B\big(\frac{n}{3\ln n}\big)$ and $\partial B(n/3)$
up to time~$t_\alpha$. Now, we estimate the (conditional) probability 
that an excursion hits the set~$A$. For this, observe that 
Lemmas~\ref{l_exit_balls}, 
\ref{l_hit_A} and~\ref{l_relation_S_hatS} imply that,
for any $x\in \partial B\big(\frac{n}{3\ln n}\big)$
\begin{align} 
\lefteqn{\hP_x\big[\htau_1(A)>\htau_1(\partial B(n/3))\big]}
\nonumber\\
&= \frac{P_x\big[\tau_1(A)>\tau_1(\partial B(n/3)), 
\tau_1(0)> \tau_1(\partial B(n/3))\big]}{P_x\big[\tau_1(0)>
\tau_1(\partial B(n/3))\big]} \big(1+O((n\ln n)^{-1}))\big)
\nonumber\\ \nonumber
 & = \frac{a(x)-\capa(A)+O(\frac{\ln^2 n}{n})}
{a(n/3)-\capa(A) +O(\frac{\ln^2 n}{n})}
\times \frac{a(n/3)+O(n^{-1})}{a(x)}
  \big(1+O((n\ln n)^{-1}))\big)
 \\ \nonumber
 & = \frac{1-\frac{\capa(A)}{a(x)}}
{1-\frac{\capa(A)}{a(n/3)}}
 \big(1+O(n^{-1}\ln n)\big)\\ \label{eq:decadix}
 & = 1 - \frac{\pi}{2}\capa(A)\frac{\ln\ln n}{\ln^2 n}
 \big(1+o(1)\big).
\end{align}
Note that the above is for $\s$-excursions; we still need
to transfer this result to the conditioned random walk
on the torus.

Recall the notation $\Gamma^{(x)}_{0,R}$ from the 
beginning of Section~\ref{s_aux_hat_s}.
Then,
for a fixed $x \in \partial B\big(\frac{n}{3\ln n}\big)$
 let us define the set of paths
\[
 \Lambda_j =  \big\{\vr \in \Gamma^{(x)}_{0,n/3} :
  (j-1)n^2<|\vr|\leq jn^2\big\}.
\] 
It is straightforward to obtain that (since, regardless of the 
starting position, after $O(n^2)$ steps the walk goes out 
of~$B(n/3)$ with uniformly positive probability)
\begin{equation}
\label{sortir_boule}
 \max\big(P_x[\Lambda_j],\hP_x[\Lambda_j]\big) \leq e^{-cj}
\end{equation}
for some $c>0$.

To extract from~\eqref{eq:decadix} the corresponding formula for the 
$\tX$-excursion, we first observe that, for 
$x\in \partial B\big(\frac{n}{3\ln n}\big)$ and $s\geq n^2\sqrt{\ln n}$
\begin{align}
 h(s,x) &= P_x[T_n(0)>T_n(\partial B(n/3))]\times
P_x[T_n(0)> s \mid T_n(0)>T_n(\partial B(n/3))] \nonumber\\
&\qquad  \qquad \qquad \qquad+ P_x[T_n(\partial B(n/3))\geq T_n(0)>s]
\nonumber\\
&= \frac{a(x)}{a(n/3)+O(n^{-1})}P_x[T_n(0)> s
\mid T_n(0)>T_n(\partial B(n/3))]  
+ \psi_{x,s,n}\nonumber\\
& = \frac{a(x)}{a(n/3)+O(n^{-1})} 
\sum_{\substack{ y\in\partial B(n/3), \\k\geq 1}} h(s-k,y)
\ell_{y,k}
 + \psi_{x,s,n},
 \label{transfer_h}
\end{align}
where $\ell_{y,k}=P_x\big[X_{T_n(\partial B(n/3))}=y,
T_n(\partial B(n/3))=k  \mid T_n(0)>T_n(\partial B(n/3))\big]$
and $\psi_{x,s,n}= P_x[T_n(\partial B(n/3))\geq T_n(0)>s]$.
Clearly, by~\eqref{sortir_boule}, 
$\psi_{x,s,n} \leq \max_x P_x[T_n(\partial B(n/3))>s] 
\leq e^{-Cs/n^2}$ for some $C>0$.

Also, we need the following fact: there exist $c_6, c'_6>0$
such that 
\begin{equation}
\label{h_lower}
 h(s,x) \geq \frac{c'_6}{\ln n} \exp\Big(-\frac{c_6 s}{n^2\ln n}\Big)
\end{equation}
for all~$s$ and all~$x\in\Z^2_n\setminus\{0\}$.
To prove~\eqref{h_lower}, it is enough to observe that
\begin{itemize}
 \item  a particle 
   starting from~$x$ will reach $\partial B\big(\frac{n}{3\ln n}\big)$
   without hitting~$0$ with probability at least $O\big(\frac{1}{\ln n}\big)$;
 \item the number of (possibly incomplete) excursions 
between $\partial B\big(\frac{n}{3\ln n}\big)$ and $\partial B(n/3)$ until 
time~$s$ does not exceed $\lceil\frac{3s}{n^2 \ln\ln n}\rceil$
with at least constant probability
by Lemma~\ref{l_excursions_torus}; and
 \item regardless of the past, each excursion hits~$0$
 with probability $\frac{\ln\ln n}{\ln n}(1+o(1))$, by (\ref{nothit_0}).
\end{itemize}
Observe that Lemma~\ref{l_reg_h} and~\eqref{h_lower} imply that 
for any $y,y'\in\partial B(n/3)$ and any $t,r\geq 0$
(in the following, $\nu(\cdot) = P_{y'}[X_t=\cdot\mid T_n(0)>t]$)
\begin{equation}
\label{upper_h_over_h}
 \frac{h(t,y)}{h(t+r,y')} = \frac{h(t,y)}{h(t,y')h(r,\nu)}
 \leq c'' \ln n \times \exp\Big(\frac{c_6 r}{n^2\ln n}\Big).
\end{equation}

Now, going back to~\eqref{transfer_h} 
{and setting 
$a_j^{(n)}  = \frac{c_1j}{\ln n}$ with~$c_1$ from Lemma~\ref{l_reg_h}}, observe that
for \emph{any} $y_0\in\partial B(n/3)$
(recall that $t_\alpha = \frac{4\alpha}{\pi}n^2\ln^2 n$)
\begin{align}
 \lefteqn{\sum_{\substack{y\in\partial B(n/3),\\1\leq k\leq t_\alpha}} 
h(s-k,y)
\ell_{y,k}}\nonumber\\
&= \sum_{1\leq j \leq \sigma_1\ln n}
\sum_{\substack{y\in\partial B(n/3),\\ (j-1)n^2<k\leq jn^2}} h(s-k,y)
\ell_{y,k}
+\sum_{\substack{y\in\partial B(n/3),\\ 
\sigma_1n^2\ln n< k \leq n^2\ln^{4/3} n }}
h(s-k,y) \ell_{y,k}
\nonumber\\
& \quad{}+\sum_{\substack{y\in\partial B(n/3),\\ 
n^2\ln^{4/3} n < k \leq  t_\alpha}}
h(s-k,y) \ell_{y,k}
\nonumber\\
&= \sum_{1\leq j \leq \sigma_1\ln n}
\sum_{\substack{y\in\partial B(n/3),\\ (j-1)n^2<k\leq jn^2}}
\ell_{y,k}h(s,y_0)
\Big(1+ O\big(a_j^{(n)}\big)\Big) 
\nonumber\\
& \quad{}+\sum_{\substack{y\in\partial B(n/3),\\ 
\sigma_1n^2\ln n < k \leq n^2\ln^{4/3} n }}
c'' \ell_{y,k}h(s,y_0)  \exp\big(c_6 \ln^{1/3} n\big)\ln n
\nonumber\\
& \quad{}+\sum_{\substack{y\in\partial B(n/3),\\ 
n^2\ln^{4/3} n < k \leq  t_\alpha}}
c'' \ell_{y,k}h(s,y_0) \exp\Big(\frac{4\alpha c_6 r \ln n}{\pi}\Big)\ln n 
\nonumber\\
&= h(s,y_0) \Bigg(1+O\Big( \sum_{j\geq 1}e^{-cj} a_j^{(n)}\Big)
+ O\big(\exp\big(-(C\sigma_1\ln n-c_6 \ln^{1/3} n)\big)\ln n\big)
\nonumber\\
& \qquad \qquad \qquad {}+ O\Big(\exp\Big(-C\ln^{4/3}n + 
\frac{4\alpha c_6 r \ln n}{\pi}\Big)\ln n\Big)\Bigg)
\nonumber\\
&=h(s,y_0) \Big(1+O\Big(\frac{1}{\ln n}\Big)\Big),
\label{regularize_sum_h}
\end{align}
due to Lemma~\ref{l_reg_h}, \eqref{sortir_boule} and~\eqref{upper_h_over_h}. 

We plug~\eqref{regularize_sum_h}
into~\eqref{transfer_h}, divide by~$h(s,y_0)$, and use~\eqref{h_lower}
to obtain, for 
$x\in \partial B\big(\frac{n}{3\ln n}\big)$ and $s\geq n^2\sqrt{\ln n}$
\[
 \frac{h(s,x)}{h(s,y_0)} = \frac{a(x)}{a(n/3)+O(n^{-1})}
 \Big(1+O\Big(\frac{1}{\ln n}\Big)\Big)
 +\psi'_{s,n},
\]
where $|\psi'_{s,n}|\leq (c'_6)^{-1}\exp\big(-\frac{s}{n^2}
\big(c-\frac{c_6}{\ln n}\big)\big)$. Equivalently, 
\begin{equation}
\label{frac_h}
 \frac{h(s,y_0)}{h(s,x)} = \frac{a(n/3)}{a(x)}
 \Big(1+O\Big(\frac{1}{\ln n}\Big)\Big).
\end{equation} 

For $A\subset \Z^2$,
let us define also the hitting times of the 
corresponding set on the torus
for the $\tX$-walk \emph{after}
a given time~$s$:  
\[
 \widetilde{T}_n^{(s)}(A) = \min\{k\geq s: \tX_k\in \Upsilon_nA\};
\]
we abbreviate $\widetilde{T}_n(A)=\widetilde{T}_n^{(0)}(A)$.
Write
\begin{align*}
 \lefteqn{\IP\big[\widetilde{T}^{(s)}_n(A)
 <\widetilde{T}^{(s)}_n(\partial B(n/3))
 \mid \tX_s=x\big]}\\
 &= \sum_{\vr} \frac{h(t_\alpha-s-|\vr|, \vr_{\text{end}})}{h(t_\alpha-s,x)}
\Big(\frac{1}{4}\Big)^{|\vr|}\\
 &= \Bigg(\sum_{\vr: \frac{|\vr|}{n^2} \leq \sqrt{\ln n}} 
 + \sum_{\vr: \sqrt{\ln n}< \frac{|\vr|}{n^2} \leq \ln^{4/3} n }
  + \sum_{\vr:  \frac{|\vr|}{n^2} > \ln^{4/3} n}\Bigg)
 \frac{h(t_\alpha-s-|\vr|, \vr_{\text{end}})}{h(t_\alpha-s,x)}
\Big(\frac{1}{4}\Big)^{|\vr|},
\end{align*}
where the sums are over paths~$\vr$ that begin in~$x$,
 end on the first visit to~$\partial B(n/3)$, and touch~$A$ 
without touching~$0$.

Using Lemma~\ref{l_reg_h}, \eqref{sortir_boule} and~\eqref{frac_h} 
for the first 
sum, and dealing with the second and third sums as in the 
derivation of~\eqref{regularize_sum_h},
we obtain
(recall that the term $\hP_x\big[\htau_1(A)<\htau_1(\partial B(n/3))\big]$
is of order $\frac{\ln\ln n}{\ln^2 n}$)
\begin{align*}
 \lefteqn{\IP\big[\widetilde{T}^{(s)}_n(A)
 <\widetilde{T}^{(s)}_n(\partial B(n/3))
 \mid \tX_s=x\big]}\\
&=  \sum_{\vr} \frac{a(n/3)}{a(x)}
\Big(\frac{1}{4}\Big)^{|\vr|}\Big(1+O\Big(\frac{1}{\sqrt{\ln n}}\Big)\Big)
+O\big(\exp(-C\sqrt{\ln n})\big) + O\big(\exp(-C\ln^{-4/3} n)\big)\\
&=\hP_x\big[\htau_1(A)<\htau_1(\partial B(n/3))\big]
\Big(1+O\Big(\frac{1}{\sqrt{\ln n}}\Big)\Big);
\end{align*}
again, the sum is over all paths~$\vr$ that begin in~$x$,
 end on the first visit to~$\partial B(n/3)$, and touch~$A$ 
without touching~$0$ (observe that the last equality comes 
   from the definition of~$\s$).
So, using (\ref{eq:decadix}), we obtain
for all $s\leq t_\alpha - n^2\sqrt{\ln n}$ 
and all $x \in \partial B\big(\frac{n}{3 \ln n}\big)$,
\begin{equation}
\label{exc_tilde_hitsA}
 \IP\big[\widetilde{T}^{(s)}_n(A)>\widetilde{T}^{(s)}_n(\partial B(n/3))
 \mid \tX_s=x\big] 
= 1-\frac{\pi}{2}\capa(A)\frac{\ln\ln n}{\ln^2 n}(1+o(1)).
\end{equation}

Before we are able to conclude the proof of Theorem~\ref{t_conditional}, we 
need another step to take care of times close to~$t_\alpha$. 
Consider any $x\in\partial B\big(\frac{n}{3\ln n}\big)$
and any~$s\geq 0$ such that $t_\alpha-s\leq n^2\sqrt{\ln n}$. 
Then, \eqref{df_h_transf} and~\eqref{h_lower} 
together with the fact that $h(\cdot, \cdot)$
is nonincreasing with respect to the first (temporal) argument
imply that
\begin{align}
 \IP_x\big[(\tX_0,\ldots,\tX_{|\vr|})=\vr\big]
 &= \frac{h(t_\alpha-s-|\vr|, \vr_\text{end})}{h(t_\alpha-s,x)}
 \IP_x\big[(X_0,\ldots,X_{|\vr|})=\vr\big]
 \nonumber\\
 &\leq c_8 \IP_x\big[(X_0,\ldots,X_{|\vr|})=\vr\big]
 \label{measure_comparison}
\end{align}
for any path~$\vr$ with $\vr_0=x$.
Then, similarly to Section~\ref{s_aux_torus}, define 
$\tilde{J}_k$  $\tilde{D}_k$ to be the starting and
ending times of $k$th excursion of~$\tX$
between~$\partial B\big(\frac{n}{3\ln n}\big)$ 
and~$\partial B(n/3)$, $k\geq 1$. Let
\[
 \zeta = \min\{k: \tilde{J}_k\geq t_\alpha-n^2\sqrt{\ln n}\}
\]
be the index of the first $\tX$-excursion that 
starts after time $t_\alpha-n^2\sqrt{\ln n}$. 
Let $\xi'_1,\xi'_2,\xi'_3,\ldots$ be a sequence of i.i.d.\
Bernoulli random variables independent of everything,
with
\[
\IP[\xi'_k=1]=1-\IP[\xi'_k=0] 
  =1-\frac{\pi}{2}\capa(A)\frac{\ln\ln n}{\ln^2 n}.
\]
For $k\geq 1$ define
two sequences of random variables
\[
 \xi_k = \begin{cases}
          \1{\tX_j\notin A\text{ for all } j\in[\tilde{J}_k,\tilde{D}_k]},
            & \text{ for } k<\zeta,\\
            \xi'_k, & \text{ for } k\geq \zeta,
         \end{cases}
\]
and
\[
 \eta_k = \1{\tX_j\notin A\text{ for all } j\in[\tilde{J}_{\zeta+k-1},
 \tilde{D}_{\zeta+k-1}]}.
\]

Now, observe that~\eqref{exc_tilde_hitsA} and the 
strong Markov property imply that
\begin{equation}
\label{hit_normalexcursion}
 \IP[\xi_k=1\mid \xi_1,\ldots,\xi_{k-1}] = 
 1-\frac{\pi}{2}\capa(A)\frac{\ln\ln n}{\ln^2 n}(1+o(1)).
\end{equation}
Also,
the relation~\eqref{measure_comparison}
together with Lemma~\ref{l_hit_A} imply that
\begin{equation}
\label{hit_lastexcursion}
 \IP[\eta_k=0\mid \eta_1,\ldots,\eta_{k-1}]\leq c_8 \frac{\ln\ln n}{\ln n}.
\end{equation}
Denote
\[
 \zeta' = \max\{k: \tilde{J}_k\leq t_\alpha\}.
\]
Then, \eqref{measure_comparison} and Lemma~\ref{l_excursions_torus}
imply that (note that $\pi/2<3$)
\begin{equation}
\label{number_lastexcursions}
 \IP\Big[\zeta'-\zeta\geq  \frac{3\sqrt{\ln n}}{\ln\ln n}\Big]
 \leq \exp\Big(-c_9 \frac{\sqrt{\ln n}}{\ln\ln n}\Big).
\end{equation}

Recall the notation~$\delta_{n,\alpha}$ from the 
beginning of the proof of this theorem.
We can write (recall~\eqref{df_N'a})
\begin{align*}
\lefteqn{ \IP\Big[\xi_k=1 \text{ for all }k\leq 
 (1-\delta_{n,\alpha})\frac{2\alpha\ln^2 n}{\ln\ln n}
-\frac{3\sqrt{\ln n}}{\ln\ln n}\Big]}\\
 &\geq
\IP\Big[\widetilde{T}_n(A)>t_\alpha, N'_\alpha \geq (1-\delta_{n,\alpha})
 \frac{2\alpha\ln^2 n}{\ln\ln n}, 
\zeta'-\zeta \leq\frac{3\sqrt{\ln n}}{\ln\ln n}\Big]
\end{align*}
so
\begin{align}
 \IP\big[\widetilde{T}_n(A)>t_\alpha\big] &\leq 
 \IP\Big[\xi_k=1 \text{ for all }k\leq 
 (1-\delta_{n,\alpha})\frac{2\alpha'\ln^2 n}{\ln\ln n}
-\frac{3\sqrt{\ln n}}{\ln\ln n}\Big]
 \nonumber\\
 &\qquad + \IP\Big[N'_\alpha < (1-\delta_{n,\alpha})
 \frac{2\alpha\ln^2 n}{\ln\ln n}\Big]
+\IP\Big[\zeta'-\zeta>\frac{3\sqrt{\ln n}}{\ln\ln n}\Big].
\label{main_upper}
\end{align}
Also,
\begin{align}
 \lefteqn{\IP\big[\widetilde{T}_n(A)>t_\alpha\big]}
 \nonumber\\
 &\geq
 \IP\Big[\xi_k=1 \text{ for all }k\leq 
 (1+\delta_{n,\alpha})\frac{2\alpha\ln^2 n}{\ln\ln n},
 N_\alpha\leq (1+\delta_{n,\alpha})
 \frac{2\alpha\ln^2 n}{\ln\ln n},
 \nonumber\\
&\qquad\quad 
 \eta_k=1 \text{ for all }k\leq \frac{3\sqrt{\ln n}}{\ln\ln n},
  \zeta'-\zeta \leq \frac{3\sqrt{\ln n}}{\ln\ln n},
\tX_0\notin B\Big(\frac{n}{3\ln n}\Big)\Big]
 \nonumber\\
 &\geq \IP\Big[\xi_k=1 \text{ for all }k\leq 
 (1+\delta_{n,\alpha})\frac{2\alpha\ln^2 n}{\ln\ln n}\Big]
 - \IP\Big[N_\alpha > (1+\delta_{n,\alpha})
 \frac{2\alpha\ln^2 n}{\ln\ln n}\Big]
  \nonumber\\
  &\quad -
 \IP\Big[\Big(\eta_k=1 \text{ for all }k\leq 
 \frac{3\sqrt{\ln n}}{\ln\ln n}\Big)^\complement\Big]
 -\IP\Big[\zeta'-\zeta > \frac{3\sqrt{\ln n}}{\ln\ln n}\Big]
-O\Big(\frac{1}{\ln^2 n}\Big).
\label{main_lower}
\end{align}
Using~\eqref{hit_normalexcursion},
we obtain that the first terms in the right-hand
sides of~\eqref{main_upper}--\eqref{main_lower} are both
equal to
\[
 \Big(1-\frac{\pi}{2}\capa(A)\frac{\ln\ln n}{\ln^2 n}(1+o(1))
 \Big)^{\frac{2\alpha \ln^2 n}{\ln \ln n}}
 = (1+o(1))\exp\big(-\pi\alpha\capa(A)\big).
\]
The other terms in the right-hand
sides of~\eqref{main_upper}--\eqref{main_lower} are~$o(1)$
due to~\eqref{cond_numb_exc}, \eqref{hit_lastexcursion}, and
\eqref{number_lastexcursions}.
Since, by~\eqref{df_h_transf},
\[
\IP[\Upsilon_n A \subset    U_{t_\alpha}^{(n)} 
\mid 0\in U_{t_\alpha}^{(n)}]
=\IP\big[\widetilde{T}_n(A)>t_\alpha\big],
\]
the proof of Theorem~\ref{t_conditional} is concluded.
\end{proof}

\section*{Acknowledgements}
The authors are grateful to Caio Alves and Darcy Camargo for 
pointing out to us that $\frac{1}{a}$ is a martingale for the 
conditioned walk, and to David Belius and Augusto Teixeira
for useful discussions. Darcy Camargo also did the simulations
of the model presented on Figure~\ref{f_simulation}.
The authors thank the referees for their careful
reading of the paper and many valuable comments and suggestions.

Also, the authors thank
 the financial support from Franco-Brazilian Scientific Cooperation
program.
S.P.\ and M.V. were partially supported by
CNPq (grants 300886/2008--0 and 301455/2009--0). 
The last two authors thank FAPESP (2009/52379--8,
2014/06815--9, 2014/06998--6)  
for financial support. F.C.\ is partially supported by 
 MATH Amsud  program
15MATH01-LSBS.


\end{document}